\documentclass[11pt,a4paper]{article}
\usepackage{mathrsfs}
\usepackage{multirow}
 \usepackage{epsfig}
 \usepackage{epstopdf}
\usepackage[T1]{fontenc}
\usepackage{geometry}
\usepackage{amsmath}
\usepackage{bm}
\usepackage{amsbsy,latexsym,amsfonts, epsfig, color, authblk, amssymb, graphics, bm}
\usepackage{epsf,slidesec,epic,eepic}
\usepackage{fancybox}
\usepackage{fancyhdr}
\usepackage{setspace}
\usepackage{cases}
\usepackage{nccmath}
\usepackage[colorlinks, citecolor=blue]{hyperref}

\newtheorem{theorem}{Theorem}[section]
\newtheorem{lemma}{Lemma}[section]

\newtheorem{definition}{Definition}[section]

\newtheorem{proposition}{Proposition}[section]

\newtheorem{remark}{Remark}[section]

\newtheorem{alemma}{Lemma}

\newenvironment{proof}{{\noindent \bf Proof:}}{\hfill$\Box$\medskip}
\newenvironment{aproof}{}{\hfill$\Box$\medskip}

\definecolor{lred}{rgb}{1,0.8,0.8}
\definecolor{lblue}{rgb}{0.8,0.8,1}
\definecolor{dred}{rgb}{0.6,0,0}
\definecolor{dblue}{rgb}{0,0,0.5}
\definecolor{dgreen}{rgb}{0,0.5,0.5}

 \title{Locally upper Lipschitz of the perturbed KKT systems for the Ky Fan $k$-norm matrix conic optimization problem\footnote{This work is supported by the National Natural Science Foundation of China under project No. 11571120 and No. 11501219,
 and the Natural Science Foundation of Guangdong Province under project No. 2015A030313214 and No. 2015A030310298.} }

 \author{Yulan Liu\footnote{Ylliu@gdut.edu.cn. School of Mathematics, South China University of Technology, Guangzhou.}\ \ {\rm and}\ \
 Shaohua Pan\footnote{Corresponding author (shhpan@scut.edu.cn). School of Mathematics, South China University of Technology,
 Guangzhou (510641), China.}}

 \date{February 20, 2016 (revised)}

\begin{document}

 \maketitle

 \begin{abstract}
  In this note we consider the Ky Fan $k$-norm matrix conic optimization problem,
  which includes the nuclear norm regularized minimization problem as a special case.
  For this class of nonpolyhedral matrix conic optimization problems, we show that the solution
  mappings of two types of perturbed KKT systems are locally upper Lipschitz at the origin for a KKT point,
  under the second-order sufficient condition of the stationary point and the strict Robinson's CQ
  for the associated multiplier. This result depends on an equivalent characterization
  for the point pair in the graphical derivative of the normal cone mapping of the $k$-norm matrix cone,
  and implies a local error bound which plays a crucial role in the convergence rate analysis of algorithms.


  \bigskip
  \noindent
  {\bf AMS Subject Classifictions (2010): 90C30, 90C31, 54C60}

 \end{abstract}

%----------------------------------------------------------------------------------------------Introduction
  \section{Introduction}\label{sec1}

  Throughout this note, we write $\mathbb{X}=\mathbb{R}\times\mathbb{R}^{m\times n}$
  with $m\le n$, where $\mathbb{R}^{m\times n}$ is the vector space of all $m\times n$
  real matrices endowed with the trace inner product $\langle \cdot,\cdot\rangle$ and
  its induced Frobenius norm $\|\cdot\|$. Let
  \(
    K:=\big\{(t,X)\in\mathbb{X}\ |\ \|\sigma(X)\|_{(k)}\le t\big\}
  \)
 be the Ky Fan $k$-norm matrix cone for an integer $k\in[1,m]$,
 where $\sigma(X)\in\mathbb{R}^m$ denotes the singular value vector of $X$ with nonincreasing entries,
 and $\|\cdot\|_{(k)}$ means the vector Ky Fan $k$-norm.
 Given twice continuously differentiable $f\!:\mathbb{X}\to\mathbb{R}$,
 $h\!:\mathbb{X}\to\mathbb{R}^p$ and $G\!:\mathbb{X}\to\mathbb{X}$,
 we are interested in the following Ky Fan $k$-norm matrix conic optimization problem
 \begin{equation}\label{NMCP}
   \min_{\mathcal{X}\in \mathbb{X}}\left\{f(\mathcal{X})\!:\ h(\mathcal{X})=0,\ G(\mathcal{X})\in K\right\},
 \end{equation}
 which is well-known to have a host of applications. A typical example is
 the popular nuclear norm regularized least squares problem (see, e.g., \cite{Candes09,Recht10,FPST13,CLMW11})
 of the form
 \begin{equation}\label{NLS}
   \min_{X\in\mathbb{R}^{m\times n}}\left\{\frac{1}{2}\|\mathcal{A}(X)-b\|^2+\rho\|X\|_{*}\!:\ \mathcal{E}(X)=d\right\}
 \end{equation}
 where $\mathcal{A}\!:\mathbb{R}^{m\times n}\to\mathbb{R}^p$ and $\mathcal{E}\!:\mathbb{R}^{m\times n}\to\mathbb{R}^q$
 are the linear mappings, $b\in\mathbb{R}^p$ and $d\in\mathbb{R}^q$ are
 the vectors, and $\rho>0$ is the regularizing parameter.
 Clearly, one may rewrite \eqref{NLS} as the form \eqref{NMCP} with $k=m$, $f(\mathcal{X})=\frac{1}{2}\|\mathcal{A}(X)-b\|^2+\rho t$,
 $h(\mathcal{X})=\mathcal{E}(X)-d$ and $G(\mathcal{X})=\mathcal{X}$ for $\mathcal{X}=(t,X)$.
 For other special examples, the interested reader may refer to \cite{DingC10}.

 \medskip

 Let $L\!:\mathbb{X}\times\mathbb{R}^p\times\mathbb{X}\to\mathbb{R}$ denote the Lagrange function
 of problem \eqref{NMCP}, defined by
 \[
   L(\mathcal{X};\lambda,\mathcal{Y}):=f(\mathcal{X})+\langle h(\mathcal{X}),\lambda\rangle+\langle G(\mathcal{X}),\mathcal{Y}\rangle
 \]
  where $\lambda\in\mathbb{R}^p$ and $\mathcal{Y}\in\mathbb{X}$ are the Lagrange multipliers associated to
  the equality constraint $h(\mathcal{X})=0$ and the conic constraint $G(\mathcal{X})\in K$,
  respectively. It is known that the Karush-Kuhn-Tucker (KKT) optimality conditions for problem \eqref{NMCP} take the form
 \begin{equation}\label{equa-KKT}
   L_{\mathcal{X}}'\big(\mathcal{X};\lambda,\mathcal{Y}\big)=0,\ \ h(\mathcal{X})=0\ \ {\rm and}\ \
   \mathcal{Y}\in\mathcal{N}_{K}(G(\mathcal{X}))
  \end{equation}
 where $L_{\mathcal{X}}'$ is the derivative of $L$ with respect to $\mathcal{X}$,
 and $\mathcal{N}_{K}(G(\mathcal{X}))$ is the normal cone of $K$ at $G(\mathcal{X})$
 in the sense of convex analysis. By the definition of $\mathcal{N}_{K}$, it is easy to obtain
 \begin{equation}\label{decom}
   K\ni G(\mathcal{X}) \perp \mathcal{Y}\in K^{\circ}
   \Longleftrightarrow \mathcal{Y}\in\mathcal{N}_{K}(G(\mathcal{X}))\Longleftrightarrow
   G(\mathcal{X})-\Pi_{K}(G(\mathcal{X})+\mathcal{Y})=0,
 \end{equation}
  where $K^{\circ}$ is the negative polar cone of $K$, and $\Pi_{K}\!:\mathbb{X}\to\mathbb{X}$
  is the projection operator associated to $K$. So the KKT optimality conditions in \eqref{equa-KKT}
  can be equivalently written as
 \begin{equation}\label{Psi-map}
  \Psi\big(\mathcal{X},\lambda,\mathcal{Y}\big)
  :=\left(\begin{matrix}
    L_{\mathcal{X}}'\big(\mathcal{X},\lambda,\mathcal{Y}\big)\\
    h(\mathcal{X})\\
   G(\mathcal{X})-\Pi_{K}(G(\mathcal{X})+\mathcal{Y})
   \end{matrix}\right)=0.
  \end{equation}
  In the sequel, for a given feasible point $\overline{\mathcal{X}}\in\mathbb{X}$ of problem \eqref{NMCP},
  we denote by $\mathcal{M}(\overline{\mathcal{X}})$ the set of Lagrange multipliers,
  and say that $\overline{\mathcal{X}}$ is a stationary point of problem \eqref{NMCP} if and only if
  $\mathcal{M}(\overline{\mathcal{X}})\ne \emptyset$. It is well known that if $\overline{\mathcal{X}}$ is
  a locally optimal solution to problem \eqref{NMCP}, then $\overline{\mathcal{X}}$ may not be a stationary point.
  If Robinson's constraint qualification (CQ) holds at $\overline{\mathcal{X}}$,
  then $\mathcal{M}(\overline{\mathcal{X}})\ne \emptyset$ and $\overline{\mathcal{X}}$ is a stationary point.
  When $\overline{\mathcal{X}}$ is a stationary point, we say that a Lagrange multiplier $(\overline{\lambda},\overline{\mathcal{Y}})\in\mathcal{M}(\overline{\mathcal{X}})$ satisfies the strict Robinson's CQ if
  \begin{equation}\label{SCQ}
    \left(\begin{matrix}
       h'(\overline{\mathcal{X}})\\
       G'(\overline{\mathcal{X}})
       \end{matrix}\right)\mathbb{X}
       +\left(\begin{matrix}
       0\\
       \mathcal{T}_{K}(G(\overline{\mathcal{X}}))\cap\overline{\mathcal{Y}}^{\perp}
       \end{matrix}\right)=
       \left(\begin{matrix}
       \mathbb{R}^p\\
       \mathbb{X}
       \end{matrix}\right),
  \end{equation}
 where $\mathcal{T}_{K}(G(\overline{\mathcal{X}}))$ is the contingent cone
 of $K$ at $G(\overline{\mathcal{X}})$, and $\overline{\mathcal{Y}}^{\perp}\!:=\big\{\mathcal{Z}\in\mathbb{X}:\langle \mathcal{Z},\overline{\mathcal{Y}}\rangle=0\big\}$.

 \medskip

 This note is mainly concerned with the solution mappings of the perturbed KKT systems
 $\Psi(\mathcal{X},\lambda,\mathcal{Y})=\delta$ and $\widetilde{\Psi}(\delta,\mathcal{X},\lambda,\mathcal{Y})=0$
 for $\delta=(\delta_f,\delta_h,\delta_G)\!\in\mathbb{X}\times\mathbb{R}^p\times\mathbb{X}$,
 where
\begin{equation}\label{WPsi-map}
   \widetilde{\Psi}\big(\delta,\mathcal{X},\lambda,\mathcal{Y}\big)
   \!:=\!\left(\begin{matrix}
             L_{\mathcal{X}}'\big(\mathcal{X},\lambda,\mathcal{Y}\big)-\delta_{\!f}\\
                h(\mathcal{X})-\delta_h\\
                G(\mathcal{X})-\delta_G-\Pi_{K}\big(G(\mathcal{X})-\delta_{G}+\mathcal{Y}\big)
       \end{matrix}\right).
 \end{equation}
 Notice that $\widetilde{\Psi}(\delta,\mathcal{X},\lambda,\mathcal{Y})\!=\!0$
 is the KKT system of the following perturbation of \eqref{NMCP}
 \begin{equation}\label{PNMCP}
  \min_{\mathcal{X}\in\mathbb{X}}\Big\{f(\mathcal{X})-\langle\delta_{f},\mathcal{X}\rangle\!:
  \ h(\mathcal{X})=\delta_h,\,G(\mathcal{X})\in\delta_{G}+K\Big\},
 \end{equation}
 whereas $\Psi(\mathcal{X},\lambda,\mathcal{Y})=\delta$ does not correspond to
 the KKT system of any perturbation of \eqref{NMCP}.
 Let $\mathcal{J}\!:\mathbb{X}\times\mathbb{R}^p\times\mathbb{X}\rightrightarrows\mathbb{X}\times\mathbb{R}^p\times\mathbb{X}$
 and $\widetilde{\mathcal{J}}\!:\mathbb{X}\times\mathbb{R}^p\times\mathbb{X}\rightrightarrows\mathbb{X}\times\mathbb{R}^p\times\mathbb{X}$
 be the solution mappings associated to the above two types of perturbed KKT systems, respectively, i.e.,
  \begin{equation}\label{Jmap}
  \mathcal{J}(\delta):=\Big\{(\mathcal{X},\lambda,\mathcal{Y})\in \mathbb{X}\times\mathbb{R}^{p}\times\mathbb{X}
  \ |\ \Psi\big(\mathcal{X},\lambda,\mathcal{Y}\big)=\delta\Big\}
 \end{equation}
  and
  \begin{equation}\label{WJmap}
  \widetilde{\mathcal{J}}(\delta):=\left\{(\mathcal{X},\lambda,\mathcal{Y})\in \mathbb{X}\times\mathbb{R}^{p}\times\mathbb{X}
  \ |\ \widetilde{\Psi}\big(\delta,\mathcal{X},\lambda,\mathcal{Y}\big)=0\right\}.
 \end{equation}
  Clearly, $(\overline{\mathcal{X}},\overline{\lambda},\overline{\mathcal{Y}})$ is a KKT point of \eqref{NMCP}
  iff $(0,\overline{\mathcal{X}},\overline{\lambda},\overline{\mathcal{Y}})\in{\rm gph}\,\mathcal{J}$ or
  ${\rm gph}\,\widetilde{\mathcal{J}}$, while $(\overline{\mathcal{X}},\overline{\lambda},\overline{\mathcal{Y}})$
  is a KKT point of \eqref{PNMCP} associated to $\overline{\delta}=(\overline{\delta}_f,\overline{\delta}_h,\overline{\delta}_G)
  \in\mathbb{X}\times\mathbb{R}^p\times\mathbb{X}$ iff
  $(\overline{\delta},\overline{\mathcal{X}},\overline{\lambda},\overline{\mathcal{Y}})\in {\rm gph}\,\widetilde{\mathcal{J}}$.

  \medskip

  The main contribution of this work is to establish the locally upper Lipschitz
  of the multifunctions $\mathcal{J}$ and $\widetilde{\mathcal{J}}$ at $0$ for
  $\overline{\mathcal{W}}=(\overline{\mathcal{X}},\overline{\lambda},\overline{\mathcal{Y}})
  \in\mathcal{J}(0)=\widetilde{\mathcal{J}}(0)$, under the second-order sufficient condition
  of $\overline{\mathcal{X}}$ and the strict Robinson's CQ for $(\overline{\lambda},\overline{\mathcal{Y}})$.
  As will be shown in Remark \ref{main-remark}(b), the locally upper Lipschitz of $\mathcal{J}$ at $0$
  implies that the distance from any point $(\mathcal{X},\lambda,\mathcal{Y})$ near $\overline{\mathcal{W}}$
  to the whole set of KKT points can be controlled by the KKT system residual at this point.
  This local error bound plays a key role in achieving the convergence rate of the first-order
  algorithms for \eqref{NMCP}, and especially the nuclear norm regularized least-squares problem.
  While the locally upper Lipschitz of $\widetilde{\mathcal{J}}$ is important
  in the perturbation theory of optimization. This is the main motivation of this work.

  \medskip

  We notice that Zhang and Zhang \cite{ZZ15} recently derived the locally upper Lipschitz of
  the KKT mapping for the canonical perturbation of nonlinear semidefinite programming (SDP)
  problems by the equivalent Kojima's reformulation, under the second-order sufficient condition
  and the strict Robinson's CQ. Later, Han, Sun and Zhang \cite{HSZ15} established
  the locally upper Lipschitz of the perturbed KKT system for the nonlinear SDP problem
  under the same assumption and used it to provide a sufficient condition
  to guarantee the linear convergence rate of the ADMM (alternating direction method of multipliers)
  for the convex composite quadratic SDP problem. This note is also motivated by their works
  and the wide applications of the Ky Fan $k$-norm matrix conic optimization.
  When making revisions for our manuscript, we learn that Ding, Sun and Zhang \cite{DingSZ15}
  provide an equivalent characterization for the isolated calmness (i.e., the locally upper Lipschitz)
  of the KKT solution mapping for a large class of conic optimization problems.

  \medskip

 To close this section, we present a brief summary for the notations used in this paper.
 \begin{itemize}
 \item  Let $\mathbb{S}^m$ be the vector space of all $m\times m$ real symmetric matrices,
        and $\mathbb{O}^{m\times k}$ the set of all $m\!\times k$ real matrices
        with orthonormal columns. We simplify $\mathbb{O}^{m\times m}$ as $\mathbb{O}^m$.
        For any $Z\in\mathbb{R}^{m\times n}$, $\sigma(Z)$ denotes the singular value vector
        of $Z$ whose entries are arranged in a nonincreasing order, and for any $Z\in\mathbb{S}^m$, $\lambda(Z)$ is
        the eigenvalue vector with decreasing entries. For $Z\in\mathbb{R}^{m\times n}$,
        $\mathbb{O}^{m,n}(Z)$ means the following set
         \begin{equation}\label{OmnZ}
          \mathbb{O}^{m,n}(Z):=\left\{(U,V)\in\mathbb{O}^{m}\times\mathbb{O}^{n}\ |\
          Z=U[{\rm Diag}(\sigma(Z))\ \ 0]V^{\mathbb{T}}\right\}.
          \end{equation}
          Similarly, if $Z\in\mathbb{S}^m$, then $\mathbb{O}^{m}(Z)$ denotes the set
          $\left\{P\in\mathbb{O}^{m}\ |\ Z=P{\rm Diag}(\lambda(Z))P^{\mathbb{T}}\right\}$.

 \item  For a closed set $S$ and a point $\overline{x}\in S$, $\mathcal{T}_S^{i}(\overline{x})$
         and $\mathcal{T}_S(\overline{x})$ denote the inner tangent cone and the contingent cone, respectively,
         and $\mathcal{N}_S(\overline{x})$ means the limit normal cone of $S$ at $\overline{x}$.
         When $S$ is convex, $\mathcal{N}_S(\overline{x})$ is the normal cone in the sense of convex analysis.

   \item  The $e,E$ and $I$ denote a vector, a matrix of all entries being $1$ and the unit matrix,
          respectively, whose dimensions are known from the context. For a given $Z\in\mathbb{R}^{m\times n}$
          and an index set $J\subseteq\{1,\ldots,n\}$, $Z_J$ is an $m\times |J|$ matrix consisting of those columns $Z_j$ with $j\in J$,
         and for $x\in\mathbb{R}^n$, $x_J\in\mathbb{R}^{|J|}$ is a vector consisting of $x_i$ with $i\in J$.
 \end{itemize}

%-------------------------------------------------------------------------------------------------------------Section 2
  \section{Preliminaries}\label{sec2}

  In this section, $\mathbb{Y}$ and $\mathbb{Z}$ denote the finite dimensional vector spaces equipped
  with the norm $\|\cdot\|$.
  For a multifunction $\mathcal{S}\!:\mathbb{Y}\rightrightarrows\mathbb{Z}$, two important sets
  associated with it are the domain ${\rm dom}\,\mathcal{S}:=\left\{y\in\mathbb{Y}\ |\ \mathcal{S}(y)\ne \emptyset\right\}$
  and the graph ${\rm gph}\,\mathcal{S}:=\left\{(y,z)\in\mathbb{Y}\times\mathbb{Z}\ |\ z\in\mathcal{S}(y)\right\}$.
  First, we recall the locally upper Lipschitz (see \cite{Levy96,KK02}) of
  a multifunction at a point.
%-----------------------------------------------------------------------------------------------------------
  \begin{definition}\label{Def2.2}
   A multifunction $\mathcal{S}\!:\mathbb{Y}\rightrightarrows\mathbb{Z}$ is said to
   be locally upper Lipschitz at $\overline{y}$ for $\overline{z}\in\mathcal{S}(\overline{y})$
   if there exist a constant $\mu\ge 0$ and neighborhoods $\mathcal{U}$ of $\overline{y}$ and
   $\mathcal{V}$ of $\overline{z}$ such that
   \[
     \mathcal{V}\cap\mathcal{S}(\overline{y})=\{\overline{z}\}
     \ \ {\rm and}\ \
     \mathcal{S}(y)\cap\mathcal{V}\subseteq \{\overline{z}\}+\mu\|y-\overline{y}\|\mathbb{B}_{\mathbb{Y}}
     \quad\ \forall y\in\mathcal{U},
   \]
   where $\mathbb{B}_{\mathbb{Y}}$ denotes the unit ball in the space $\mathbb{Y}$ centered at the origin.
  \end{definition}

   The locally upper Lipschitz concept of $\mathcal{S}$ at $\overline{y}$ for
   $\overline{z}\in\mathcal{S}(\overline{y})$ in Definition \ref{Def2.2} is different from
   the one defined by Robinson \cite{Robinson79}, which actually requires $\mathcal{S}(y)\cap\mathcal{V}$
   is upper Lipschitz at $\overline{y}$. The locally upper Lipschitz of $\mathcal{S}$ at $\overline{y}$
   for $\overline{z}\in\mathcal{S}(\overline{y})$ is also called the isolated calmness of $\mathcal{S}$
   at $\overline{y}$ for $\overline{z}\in\mathcal{S}(\overline{y})$ in \cite{HSZ15}.
   By Lemma \ref{equiv-LUL} in Appendix A, we have the following equivalent characterization
   for the locally upper Lipschitz of a multifunction.
%---------------------------------------------------------------------------------------------
  \begin{lemma}\label{equiv-chara-LUL}
   The locally upper Lipschitz of a multifunction $\mathcal{S}\!:\mathbb{Y}\rightrightarrows\mathbb{Z}$
   at $\overline{y}$ for $\overline{z}\in\mathcal{S}(\overline{y})$ is identified with
   the existence of a constant $\mu\ge 0$ and a neighborhood $\mathcal{V}$ of $\overline{z}$ such that
   \[
     \mathcal{V}\cap\mathcal{S}(\overline{y})=\{\overline{z}\}
     \ \ {\rm and}\ \
     \mathcal{S}(y)\cap\mathcal{V}\subseteq \{\overline{z}\}+\mu\|y-\overline{y}\|\mathbb{B}_{\mathbb{Y}}
     \quad\ \forall y\in\mathbb{Y}.
   \]
  \end{lemma}

  In addition, from \cite{KR92,Levy96} we also have the following equivalent
  characterization for the locally upper Lipschitz property of a multifunction at a point of its graph.
%------------------------------------------------------------------------------------------------------------Lemma
  \begin{lemma}\label{SN-cond}
   Let $\mathcal{S}\!:\mathbb{Y}\rightrightarrows\mathbb{Z}$ be a multifunction.
   Then $\mathcal{S}$ is locally upper Lipschitz at $\overline{y}$ for
   $\overline{z}\in\mathcal{S}(\overline{y})$ if and only if
   $D\mathcal{S}(\overline{y}|\overline{z})(0)=\{0\}$.
  \end{lemma}

  Next we recall the graphical derivative of a multifunction
  $\mathcal{S}\!:\mathbb{Y}\rightrightarrows\mathbb{Z}$ from \cite[8G]{RW98}.
%-----------------------------------------------------------------------------------------------------------
  \begin{definition}\label{Def2.1}
   Consider a multifunction $\mathcal{S}\!:\mathbb{Y}\rightrightarrows\mathbb{Z}$
   and a point $\overline{y}\!\in{\rm dom}\,\mathcal{S}$. The graphical derivative
   of $\mathcal{S}$ at $\overline{y}$ for any $\overline{z}\!\in \!\mathcal{S}(\overline{y})$
   is the mapping $D\mathcal{S}(\overline{y}|\overline{z})\!:\mathbb{Y}\rightrightarrows\!\mathbb{Z}$ defined by
   \[
      v\in D\mathcal{S}(\overline{y}|\overline{z})(u)
     \ \Longleftrightarrow\ (u,v)\in\mathcal{T}_{\rm gph\,\mathcal{S}}(\overline{y},\overline{z}).
   \]
  When $\mathcal{S}$ is single-valued at $\overline{y}$, we simplify the notation
  $D\mathcal{S}(\overline{y}|\overline{z})$ to be $D\mathcal{S}(\overline{y})$.
 \end{definition}

  If the multifunction $\mathcal{S}\!:\mathbb{Y}\rightrightarrows\mathbb{Z}$
  is implicitly defined by a (directionally differentiable) single-valued mapping,
  then we have the following result for its graphical derivative.
%------------------------------------------------------------------------------------------------------------------
  \begin{lemma}\label{Graph-dir}
   Let $\mathcal{S}\!:\mathbb{Y}\rightrightarrows\mathbb{Z}$ be a multifunction defined by
   \(
     \mathcal{S}(y):=\left\{z\in\mathbb{Z}\ |\ F(y,z)=0\right\}\!
   \)
   where $F\!:\mathbb{Y}\times\mathbb{Z}\to\mathbb{Z}$ is a single-valued mapping.
   Consider any $(\overline{y},\overline{z})\in{\rm gph}\,\mathcal{S}$. Then,
   \begin{equation*}
     D\mathcal{S}(\overline{y}|\overline{z})(u)\subseteq \Big\{v\in\mathbb{Z}:\ 0\in DF(\overline{y},\overline{z})(u,v)\Big\}
     \quad\ \forall u\in\mathbb{Y}.
   \end{equation*}
   If $F$ is directionally differentiable at $(\overline{y},\overline{z})\in{\rm gph}\,\mathcal{S}$,
   then it holds that
   \[
     D\mathcal{S}(\overline{y}|\overline{z})(u)\subseteq \Big\{v\in\mathbb{Z}:\ F'((\overline{y},\overline{z});(u,v))=0\Big\}
      \quad\ \forall u\in\mathbb{Y}.
   \]
  In particular, when $F(y,z):=y+H(z)$ for any $y\in\mathbb{Y}$ and $z\in\mathbb{Z}$
  with $H\!:\mathbb{Z}\to\mathbb{Z}$ being a single-valued mapping,
  the above two inclusions become equalities.
 \end{lemma}
 \begin{proof}
  Consider any $u\in\mathbb{Y}$ and $v\in D\mathcal{S}(\overline{y}|\overline{z})(u)$.
  Then, $(u,v)\in \mathcal{T}_{{\rm gph}\,\mathcal{S}}(\overline{y},\overline{z})$.
  By the definition of $\mathcal{T}_{{\rm gph}\,\mathcal{S}}(\overline{y},\overline{z})$,
  there exist sequences $u^n\to u,v^n\to v$ and $t_n\downarrow 0$ such that
  \[
    \overline{z}+t_nv^n\in \mathcal{S}(\overline{y}+t_nu^n)\ \Longleftrightarrow\
    F(\overline{y}+t_nu^n,\overline{z}+t_nv^n)=0\quad{\rm for\ all}\ n.
  \]
  This shows that there exist sequences $w^n\equiv 0$, $(u^n,v^n)\to (u,v)$ and $t_n\downarrow 0$
  such that $t_nw^n=F(\overline{y}+t_nu^n,\overline{z}+t_nv^n)$ for all $n$.
  Consequently, $((u,v),0)\in\mathcal{T}_{{\rm gph}\,F}((\overline{y},\overline{z}),0)$,
  i.e., $0\in DF(\overline{y},\overline{z})(u,v)$. The first inclusion follows
  by the arbitrariness of $v$ in $D\mathcal{S}(\overline{y}|\overline{z})(u)$.
  When $F$ is directionally differentiable at $(\overline{y},\overline{z})\in{\rm gph}\,\mathcal{S}$,
  $0\in DF(\overline{y},\overline{z})(u,v)$ is equivalent to $F'((\overline{y},\overline{z});(u,v))=0$
  (see \cite[Equation(6.6)]{KK02}), and the second inclusion follows.
  For the last part, it suffices to prove that
  $\big\{v\in\mathbb{Y}\!: 0\in DF(\overline{y},\overline{z})(u,v)\big\}\subseteq D\mathcal{S}(\overline{y}|\overline{z})(u)$.
  Indeed, let $v$ be an arbitrary point such that $0\in DF(\overline{y},\overline{z})(u,v)$.
  Then $((u,v),0)\in \mathcal{T}_{{\rm gph}\,F}((\overline{y},\overline{z}),0)$,
  which means that there exist sequences $w^n\to 0,(u^n,v^n)\to (u,v)$ and $t_n\downarrow 0$ such that
  \[
    t_nw^n=F(\overline{y}+t_nu^n,\overline{z}+t_nv^n)=\overline{y}+t_nu^n+H(\overline{z}+t_nv^n)
    \quad {\rm for\ all}\ n.
  \]
  Consequently, $-\overline{y}-t_n(u^n-w^n)=H(\overline{z}+t_nv^n)$, i.e.,
  $\overline{z}+t_nv^n\in\mathcal{S}(\overline{y}+t_n(u^n-w^n))$.
  This shows that $(u,v)\in\mathcal{T}_{{\rm gph}\,\mathcal{S}}(\overline{y},\overline{z})$
  or equivalently $v\in D\mathcal{S}(\overline{y}|\overline{z})(u)$.
  The desired inclusion then follows by the arbitrariness of $v$ in the set
  $\{v\in\mathbb{Z}\!: 0\in DF(\overline{y},\overline{z})(u,v)\}$.
  \end{proof}

  The following lemma characterizes a relation between the graphical derivative of
  the normal cone multifunction $\mathcal{N}_{K}\!:\mathbb{X}\rightrightarrows\mathbb{X}$
  and the directional derivative of $\Pi_{K}$.
%----------------------------------------------------------------------------------------------lemma
 \begin{lemma}\label{tcone-lemma}
  Let $(\mathcal{X},\mathcal{Y})\in{\rm gph}\mathcal{N}_{K}$ be given. Then,
  for any $\mathcal{Z}_1,\mathcal{Z}_2\in\mathbb{X}$, we have that
  \begin{align*}
   (\mathcal{Z}_1,\mathcal{Z}_2)\in\mathcal{T}_{{\rm gph}{\mathcal{N}}_{K}}^{i}(\mathcal{X},\mathcal{Y})
   &\Longleftrightarrow(\mathcal{Z}_1,\mathcal{Z}_2)\in\mathcal{T}_{{\rm gph}\mathcal{N}_{K}}(\mathcal{X},\mathcal{Y})
   \Longleftrightarrow (\mathcal{Z}_2,\mathcal{Z}_1)\in \mathcal{T}_{{\rm gph}\mathcal{N}_{K^{\circ}}}(\mathcal{Y},\mathcal{X})\nonumber\\
   &\Longleftrightarrow  \Pi'_{K}\big(\mathcal{X}+\mathcal{Y};\mathcal{Z}_1+\mathcal{Z}_2\big)=\mathcal{Z}_1\nonumber\\
   &\Longleftrightarrow  \Pi'_{K^{\circ}}\big(\mathcal{X}+\mathcal{Y};\mathcal{Z}_1+\mathcal{Z}_2\big)=\mathcal{Z}_2,\nonumber
   \qquad\qquad\qquad\nonumber
  \end{align*}
  where $\mathcal{T}_{{\rm gph}\mathcal{N}_{K}}^{i}(\mathcal{X},\mathcal{Y})$ denotes the inner tangent cone
  of ${\rm gph}\mathcal{N}_{K}$ at the point $(\mathcal{X},\mathcal{Y})$.
  \end{lemma}
  \begin{proof}
  By the Moreau's decomposition theorem \cite{Moreau65},
  \(
    \Pi_{K}\big(\mathcal{X}+\mathcal{Y}\big)=\mathcal{X}+\mathcal{Y} - \Pi_{K^{\circ}}\big(\mathcal{X}+\mathcal{Y}\big).
  \)
  This implies the third equivalence. The second equivalence is immediate
  by the relation between ${\rm gph}{\mathcal{N}}_{K}$ and ${\rm gph}{\mathcal{N}}_{K^{\circ}}$.
  Thus, it suffices to establish the first and the third equivalence. For this purpose, we first establish the following implication:
  \begin{equation}\label{implication1}
    (\mathcal{Z}_1,\mathcal{Z}_2)\in {\rm gph}\,D\mathcal{N}_{K}\big(\mathcal{X}|\mathcal{Y}\big)
   \ \Longrightarrow\ \Pi'_{K}\big(\mathcal{X}+\mathcal{Y};\mathcal{Z}_1+\mathcal{Z}_2\big)=\mathcal{Z}_1.
  \end{equation}
  Let $(\mathcal{Z}_1,\mathcal{Z}_2)\in{\rm gph}\,D\mathcal{N}_{K}\big(\mathcal{X}|\mathcal{Y}\big)
  =\mathcal{T}_{{\rm gph}\mathcal{N}_{K}}(\mathcal{X},\mathcal{Y})$.
  Then there exist sequences $t_n \downarrow 0$ and $(\mathcal{Z}^n_1,\mathcal{Z}^n_2)\rightarrow(\mathcal{Z}_1,\mathcal{Z}_2)$
  such that $(\mathcal{X},\mathcal{Y})+t_n(\mathcal{Z}^n_1,\mathcal{Z}^n_2)\in{\rm gph}\mathcal{N}_{K}$.
  Recall that
  \begin{align}\label{gphNK}
  {\rm gph}\mathcal{N}_{K}
  &= \Big\{(\mathcal{X}, \mathcal{Y}) \in K \times K^\circ\!:\ \Pi_{K}(\mathcal{X}+\mathcal{Y})=\mathcal{X}\Big\}\nonumber\\
  &= \Big\{(\mathcal{X}, \mathcal{Y}) \in K \times K^\circ\!:\ \Pi_{K^\circ}(\mathcal{X}+\mathcal{Y})=\mathcal{Y}\Big\}\nonumber\\
  &=\Big\{(\mathcal{X}, \mathcal{Y}) \in K \times K^\circ\!:\ \langle \mathcal{X}, \mathcal{Y}\rangle=0\Big\}.
 \end{align}
  So,
  \(
    \Pi_{K}\big(\mathcal{X}+\mathcal{Y}+t_n(\mathcal{Z}^n_1+\mathcal{Z}^n_2)\big)=\mathcal{X} + t_n\mathcal{Z}^n_1.
  \)
  Along with $\Pi_{K}(\mathcal{X}+\mathcal{Y})=\mathcal{X}$, it follows that
 \[
   \frac{1}{t_n}\big[\Pi_{K}(\mathcal{X}+\mathcal{Y}+t_n(\mathcal{Z}^n_1+\mathcal{Z}^n_2))
     -\Pi_{K}(\mathcal{X}+\mathcal{Y}) \big]
   = \mathcal{Z}^n_1\quad {\rm for\ all}\ n.
 \]
 Since $\Pi_{K}(\cdot,\cdot)$ is directionally differentiable in the Hadamard sense,
 taking the limit $n\rightarrow +\infty$ to the both sides of the last equality yields that
 \(
  \Pi'_{K}\big(\mathcal{X}+\mathcal{Y};\mathcal{Z}_1+\mathcal{Z}_2\big) =\mathcal{Z}_1.
 \)
 This shows that the implication in \eqref{implication1} holds. Next we establish the implication that
 \begin{equation}\label{implication2}
   \Pi'_{K}\big(\mathcal{X}+\mathcal{Y};\mathcal{Z}_1+\mathcal{Z}_2\big)=\mathcal{Z}_1  \ \Longrightarrow\ (\mathcal{Z}_1,\mathcal{Z}_2)\in\mathcal{T}_{{\rm gph}\mathcal{N}_{K}}^{i}(\mathcal{X},\mathcal{Y}).
 \end{equation}
 To this end, let $\{t_n\}$ be an arbitrary sequence with $t_n\downarrow 0$, and for each $n$ define
 \[
  \big(\mathcal{X}^n,\mathcal{Y}^n\big) :=\big(\Pi_{K}(\mathcal{X}+\mathcal{Y}+t_n(\mathcal{Z}_1+\mathcal{Z}_2)),
  \Pi_{K^{\circ}}(\mathcal{X}+\mathcal{Y}+t_n(\mathcal{Z}_1+\mathcal{Z}_2) \big)\in {\rm gph}\mathcal{N}_{K}.
 \]
 Clearly, $(\mathcal{X}^n,\mathcal{Y}^n) \rightarrow (\mathcal{X},\mathcal{Y})$.
 Also, by $\Pi_{K}(\mathcal{X}+\mathcal{Y})=\mathcal{X}$ and $\Pi_{K^{\circ}}(\mathcal{X}+\mathcal{Y})=\mathcal{Y}$,
 we have
 \begin{align*}
  &\mathcal{X}^n - \mathcal{X} = \Pi_{K}(\mathcal{X}+\mathcal{Y}+t_n (\mathcal{Z}_1+\mathcal{Z}_2))- \Pi_{K}(\mathcal{X}+\mathcal{Y}),\\
  &\mathcal{Y}^n - \mathcal{Y} = \Pi_{K^{\circ}}(\mathcal{X}+\mathcal{Y}+t_n (\mathcal{Z}_1+\mathcal{Z}_2))- \Pi_{K^{\circ}}(\mathcal{X}+\mathcal{Y}).
 \end{align*}
 Together with the directional differentiability of $\Pi_{K}$ and $\Pi_{K^{\circ}}$,
 it follows that
 \begin{align*}
  \lim_{n\to\infty}\frac{(\mathcal{X}^n,\mathcal{Y}^n) - (\mathcal{X},\mathcal{Y})}{t_n}
  =\big(\Pi'_{K}(\mathcal{X}+\mathcal{Y};\mathcal{Z}_1+\mathcal{Z}_2),\Pi'_{K^{\circ}}(\mathcal{X}+\mathcal{Y};\mathcal{Z}_1+\mathcal{Z}_2)\big)
   = \big(\mathcal{Z}_1,\mathcal{Z}_2\big).
 \end{align*}
 This shows that $(\mathcal{Z}_1,\mathcal{Z}_2)\in \mathcal{T}^i_{{\rm gph}\mathcal{N}_{K}}(\mathcal{X},\mathcal{Y})$.
 Notice that $\mathcal{T}_{{\rm gph}\mathcal{N}_{K}}^{i}(\mathcal{X},\mathcal{Y})\subseteq\mathcal{T}_{{\rm gph}\mathcal{N}_{K}}(\mathcal{X},\mathcal{Y})= {\rm gph}\,D\mathcal{N}_{K}\big(\mathcal{X}|\mathcal{Y}\big)$.
 The first and the third equivalence follow from \eqref{implication1} and \eqref{implication2}.
 \end{proof}

 \medskip

  By \cite[Proposition 4.3]{DingC10} the set $K$ is second-order regular
 (see \cite[Definition 3.85]{BS00} for the definition).
  Next we recall from \cite[Theorem 3.45 \& Theorem 3.86]{BS00} the second-order
  optimality condition of problem \eqref{NMCP} at $\overline{\mathcal{X}}$, where
  $\sigma(\cdot,\mathcal{T}_K^2(G(\overline{\mathcal{X}}),G'(\overline{\mathcal{X}})\mathcal{Z}))$
  is the support function of $\mathcal{T}_K^2(G(\overline{\mathcal{X}}),G'(\overline{\mathcal{X}})\mathcal{Z})$,
  the outer second-order tangent set of $K$ at $G(\overline{\mathcal{X}})$ in the direction
  $G'(\overline{\mathcal{X}})\mathcal{Z}$, and $\mathcal{C}(\overline{\mathcal{X}})$ denotes
  the critical cone of problem \eqref{NMCP} at $\overline{\mathcal{X}}$ with the form
  \begin{equation}\label{Critical-cone-prob1}
    \mathcal{C}(\overline{\mathcal{X}}):=\big\{\mathcal{Z}\in\mathbb{X}\ |\
    h'(\overline{\mathcal{X}})\mathcal{Z}=0,\,G'(\overline{\mathcal{X}})\mathcal{Z}\in\mathcal{T}_{K}(G(\overline{\mathcal{X}})),\,
    \langle f'(\overline{\mathcal{X}}),\mathcal{Z}\rangle=0\big\}.
  \end{equation}

  \vspace{-0.3cm}

 %--------------------------------------------------------------------------------------------------
 \begin{lemma}\label{SOSC}
  Suppose that $\overline{\mathcal{X}}=(\overline{t},\overline{X})$ is a locally optimal solution
  of problem \eqref{NMCP} and Robinson's CQ holds at $\overline{\mathcal{X}}$. Then the following
  second-order necessary condition holds:
  \[
    \sup_{(\lambda,\mathcal{Y})\in \mathcal{M}(\overline{\mathcal{X}})}
    \left\{\big\langle \mathcal{Z},\nabla_{\mathcal{X}\mathcal{X}}^2L(\overline{\mathcal{X}};\lambda,\mathcal{Y})\mathcal{Z}\big\rangle
    -\sigma\big(\mathcal{Y},\mathcal{T}_K^2(G(\overline{\mathcal{X}}),G'(\overline{\mathcal{X}})\mathcal{Z})\big)\right\}\ge 0
    \quad \forall\mathcal{Z}\in\mathcal{C}(\overline{\mathcal{X}}),
  \]
  where for any $\lambda\in\mathbb{R}^p$ and $\mathcal{Y}\in\mathbb{X}$,
  $\nabla_{\mathcal{X}\mathcal{X}}^2L(\overline{\mathcal{X}};\lambda,\mathcal{Y})$ is
  the Hessian of $L(\cdot; \lambda,\mathcal{Y})$ at $\overline{\mathcal{X}}$.
  Conversely, suppose that $\overline{\mathcal{X}}$ is a stationary point of problem \eqref{NMCP}
  and Robinson's CQ holds at $\overline{\mathcal{X}}$. Then the following second-order
  sufficient condition
  \[
    \sup_{(\lambda,\mathcal{Y})\in \mathcal{M}(\overline{\mathcal{X}})}
    \left\{\big\langle \mathcal{Z},\nabla_{\mathcal{X}\mathcal{X}}^2L(\overline{\mathcal{X}};\lambda,\mathcal{Y})\mathcal{Z}\big\rangle
    -\sigma\big(\mathcal{Y},\mathcal{T}_K^2(G(\overline{\mathcal{X}}),G'(\overline{\mathcal{X}})\mathcal{Z})\big)\right\}>0
    \quad\ \forall\mathcal{Z}\in\mathcal{C}(\overline{\mathcal{X}})\backslash\{0\},
  \]
  is necessary and sufficient for the quadratic growth of problem \eqref{NMCP}
  at $\overline{\mathcal{X}}$.
 \end{lemma}

  By \cite[Proposition 4.2]{DingC10} the sigma term in Lemma \ref{SOSC} is computable.
  To introduce its expression,
  for any given $X\in\mathbb{R}^{m\times n}$, we assume that $X$ admits the SVD of the form
 \begin{equation}\label{X-SVD}
   X=\overline{U}\big[{\rm Diag}(\sigma(X))\ \ 0\big]\overline{V}^{\mathbb{T}}
   =\overline{U}\big[{\rm Diag}(\sigma(X))\ \ 0\big]\big[\overline{V}_1\ \ \overline{V}_2\big]^{\mathbb{T}},
  \end{equation}
  where $\overline{U}\in\mathbb{O}^m$ and $\overline{V}=\big[\overline{V}_1\ \ \overline{V}_2\big]\in\mathbb{O}^{n}$
  with $\overline{V}_1\in\mathbb{O}^{n\times m}$ and $\overline{V}_2\in\mathbb{O}^{n\times (n-m)}$.
  Define
 \begin{equation}\label{abc}
   a\!:=\!\big\{i\ | \sigma_i(X)>0,\ 1\le\!i\le m\big\},
   b\!:=\!\big\{i\ | \sigma_i(X)\!=0,\ 1\le\!i\le m\big\},
   c\!:=\!\{m+\!1,\ldots,n\}.
 \end{equation}
 We use $\nu_1>\nu_2>\cdots>\nu_r$ to denote the nonzero distinct singular values of $X$, and write
 \begin{equation}\label{al}
   a_l\!:=\left\{i\ |\ \sigma_i(X)=\nu_l,\ 1\le i\le m\right\}\ \ {\rm for}\ \
   l=1,2,\ldots,r\ \ {\rm and}\ \ a_{r+1}\!:=b.
 \end{equation}

 \vspace{-0.3cm}
%--------------------------------------------------------------------------------------------------
 \begin{lemma}\label{Lemma-Upsilon}
  Let $((\overline{t},\overline{X}),(\overline{\zeta},\overline{\Gamma}))\in{\rm gph}\mathcal{N}_{K}$
  be given. For any $(\tau,H)\in\mathbb{R}\times\mathbb{R}^{m\times n}$, write
  \[
    \Upsilon_{(\overline{t},\overline{X})}\big((\overline{\zeta},\overline{\Gamma}),(\tau,H)\big)
    =\sigma\big((\overline{\zeta},\overline{\Gamma}),\mathcal{T}_K^2((\overline{t},\overline{X}),(\tau,H))\big).
  \]
  If $\overline{X}=0$, then $\Upsilon_{(\overline{t},\overline{X})}\big((\overline{\zeta},\overline{\Gamma}),(\tau,H)\big)\equiv0$.
  If $\overline{X}\ne 0$, by letting $\overline{\nu}_1>\overline{\nu}_2>\cdots>\overline{\nu}_{r_0}$
  be the first $r_0$ nonzero distinct singular values of $\overline{X}$, writing
  $\beta:=\left\{i\ |\ \sigma_i(\overline{X})=\sigma_{k}(\overline{X}):=\overline{\nu}\right\}$,
  and supposing that $X=\overline{X}+\overline{\Gamma}$ has the SVD as in \eqref{X-SVD} with
  the index sets $a,\,b,\,c$ and $a_l\ (l=1,2,\ldots,r)$ given by \eqref{abc}-\eqref{al},
  then for any $(\tau,H)\in\mathbb{R}\times\mathbb{R}^{m\times n}$ we have that
  \begin{align*}
   \Upsilon_{(\overline{t},\overline{X})}\big((\overline{\zeta},\overline{\Gamma}),(\tau,H)\big)
   &:=-\zeta\sum_{j=1}^{r_0}{\rm tr}\Big(2\overline{P}_{\!a_j}^{\mathbb{T}}
           \big[\mathcal{B}(H)(\mathcal{B}(\overline{X})-\overline{\nu}_jI)^\dag\mathcal{B}(H)\big]\overline{P}_{\!a_j}\Big)\nonumber\\
   &\qquad +\Big\langle\Sigma_{\beta\beta}(\overline{\Gamma}),
       2\overline{P}_{\!\beta}^{\mathbb{T}}\mathcal{B}(H)(\mathcal{B}(\overline{X})-\overline{\nu}I)^{\dag}\mathcal{B}(H)\overline{P}_{\!\beta}\Big\rangle
  \end{align*}
  if $\sigma_{k}(\overline{X})>0$, and otherwise
  \begin{align*}
   \Upsilon_{(\overline{t},\overline{X})}\big((\overline{\zeta},\overline{\Gamma}),(\tau,H)\big)
   &:=-\zeta\sum_{j=1}^{r_0}{\rm tr}\Big(2\overline{P}_{\!a_j}^{\mathbb{T}}
           \big[\mathcal{B}(H)(\mathcal{B}(\overline{X})-\overline{\nu}_jI)^\dag\mathcal{B}(H)\big]\overline{P}_{\!a_j}\Big)\nonumber\\
   &\qquad +\Big\langle\big[\Sigma_{\beta\beta}(\overline{\Gamma})\ \ 0\big],
       2\big[\overline{U}_{\beta}^{\mathbb{T}}H\overline{X}^\dag H\overline{V}_{\beta}\quad \overline{U}_{\beta}^{\mathbb{T}}H\overline{X}^\dag H\overline{V}_2\big]\Big\rangle,
  \end{align*}
  where $X^{\dag}$ denotes the Moore-Penrose pseduo-inverse of $X$,
  $\mathcal{B}:\mathbb{R}^{m\times n}\to\mathbb{S}^{m+n}$ is a linear mapping
  defined by
   \(
    \mathcal{B}(Z)=\left[\begin{matrix}
                        0 & Z\\
                        Z^{\mathbb{T}}& 0
                   \end{matrix}\right]
  \)
  for $Z\in\mathbb{R}^{m\times n}$, and $\overline{P}\in\mathbb{O}^{m+n}$ is defined by
  \[
     \overline{P}=\frac{1}{\sqrt{2}}\left[\begin{matrix}
       \overline{U}_{a} & \overline{U}_{b} & 0 & \overline{U}_{b} & \overline{U}_{a}^{\uparrow}\\
       \overline{V}_{a} & \overline{V}_{b} & \sqrt{2}\overline{V}_2 & -\overline{V}_{b} & -\overline{V}_{a}^{\uparrow}
       \end{matrix}\right].
   \]
  \end{lemma}

%-------------------------------------------------------------------------------------------Section 2
  \section{Main results}\label{sec3}

  We first establish a proposition to provide an equivalent characterization for the point pair in
  ${\rm gph}\,D\mathcal{N}_{K^\circ}((\overline{\zeta},\overline{\Gamma})|(\overline{t},\overline{X}))$.
  The proof of this proposition is put in Appendix B.

%-------------------------------------------------------------------------------------------------------Proposition
 \begin{proposition}\label{main-prop}
  Let $((\overline{t},\overline{X}),(\overline{\zeta},\overline{\Gamma}))\!\in{\rm gph}\mathcal{N}_{K}$ be given.
  Write $\mathcal{X}:=(t,X)=(\overline{t},\overline{X})+(\overline{\zeta},\overline{\Gamma})$.
  Let $X$ have the SVD as in \eqref{X-SVD} with $a,b,c$ and $a_l\ (l=1,2,\ldots,r)$
  defined by \eqref{abc} and \eqref{al}. Then,
  $((\Delta t,\Delta X),(\Delta\zeta,\Delta\Gamma))\in{\rm gph}D\mathcal{N}_{K^\circ}
  \big((\overline{\zeta},\overline{\Gamma})|(\overline{t},\overline{X})\big)$ if and only if
  \begin{subnumcases}{}\label{equa1-main-prop}
   (\Delta\zeta,\Delta\Gamma)\in\mathcal{C}_K(\mathcal{X}),\\
   \label{equa2-main-prop}
   (\Delta t,\Delta X)-\big(0,\overline{U}\mathfrak{X}(\overline{U}^{\mathbb{T}}\Delta\Gamma\overline{V})\overline{V}^{\mathbb{T}}\big)
   \in [\mathcal{C}_K(\mathcal{X})]^{\circ},\\
   \label{equa3-main-prop}
   \big\langle(\Delta t,\Delta X),(\Delta\zeta,\Delta\Gamma)\big\rangle
   =-\Upsilon_{(\overline{t},\overline{X})}\big((\overline{\zeta},\overline{\Gamma}),(\Delta\zeta,\Delta\Gamma)\big),
  \end{subnumcases}
  where $\mathcal{C}_K(\mathcal{X})=\mathcal{T}_K((\overline{t},\overline{X}))\cap(\overline{\zeta},\overline{\Gamma})^{\perp}$
  is the critical cone associated to the complementarity problem $K\ni(t,X)\perp(\tau,Y)\in K^{\circ}$,
  and $\mathfrak{X}$ is a mapping with $\mathfrak{X}(\Delta\widetilde{\Gamma})=0$ if $\mathcal{X}\in{\rm int}\,K$
  or ${\rm int}\,K^{\circ}$, and otherwise $\mathfrak{X}(\Delta\widetilde{\Gamma})$ given by \eqref{MFX-Case1}
  if $\sigma_k(\overline{X})>0$ and given by \eqref{MFX-Case2} if $\sigma_k(\overline{X})=0$.
 \end{proposition}

  Now we are in a position to establish the main results, which are stated as follows.
%********************************************************************************************************Theorem
 \begin{theorem}
  Let $\overline{\mathcal{X}}\!=(\overline{t},\overline{X})\in\mathbb{X}$ be a stationary point of \eqref{NMCP}
  with $(\overline{\lambda},\overline{\mathcal{Y}})\in\mathcal{M}(\overline{\mathcal{X}})$.
  Write $\overline{\mathcal{Y}}=(\overline{\tau},\overline{Y})$ and
  $G(\overline{\mathcal{X}})=(G_1(\overline{\mathcal{X}}),G_2(\overline{\mathcal{X}}))$.
  Let $X=\overline{Y}+G_2(\overline{\mathcal{X}})$ have the SVD as in \eqref{X-SVD}
  with $a,b,c$ and $a_l\ (l=1,\ldots,r)$ given by \eqref{abc} and \eqref{al} when
  $X\notin{\rm int}\,K\cup{\rm int}\,K^{\circ}$.
  \begin{itemize}
   \item[{\bf (a)}] If the second-order sufficient condition of problem \eqref{NMCP} holds
              at $\overline{\mathcal{X}}$ and $(\overline{\lambda},\overline{Y})$ satisfies
              the strict Robinson's CQ, then the multifunctions $\mathcal{J}$ and
              $\widetilde{\mathcal{J}}$ defined in equations \eqref{Jmap} and \eqref{WJmap}
              respectively are locally upper Lipschitz at the origin
              for $(\overline{\mathcal{X}},\overline{\lambda},\overline{\mathcal{Y}})$.

   \item[{\bf (b)}] If the multifunction $\mathcal{J}$ is locally upper Lipschitz at the origin
                    for $(\overline{\mathcal{X}},\overline{\lambda},\overline{\mathcal{Y}})$,
                    then the strict Robinson's CQ holds at $(\overline{\lambda},\overline{\mathcal{Y}})$.
   \end{itemize}
 \end{theorem}
 \begin{proof}
  {\bf (a)} Since $\Pi_{K}(\cdot)$ is directionally differentiable everywhere and globally
  Lipschitz continuous by \cite{DingC10}, and the function $f$ and the mappings $h$ and $G$ are twice continuously differentiable,
  by \cite[Proposition 2.47]{BS00} the mapping $\widetilde{\Psi}$ is directionally differentiable with
  \begin{align}\label{Fdir}
   &\widetilde{\Psi}'\big((\delta,\mathcal{X},\lambda,\mathcal{Y}\big);(\Delta_{\delta},\Delta_{\mathcal{X}},\Delta_{\lambda},\Delta_{\mathcal{Y}})\big)\nonumber\\
   =&\left(\begin{matrix}
      -\Delta_{\delta_f}+\nabla_{\mathcal{X}\mathcal{X}}^2L(\mathcal{X},\lambda,\mathcal{Y})\Delta_{\mathcal{X}}
      +[h'(\mathcal{X})]^*\Delta_{\lambda}+[G'(\mathcal{X})]^*\Delta_{\mathcal{Y}}\\
       -\Delta_{\delta_h}+h'(\mathcal{X})\Delta_{\mathcal{X}}\\
      G'(\mathcal{X})\Delta_{\mathcal{X}}-\Delta_{\delta_{G}}-\Pi_{K}'\big(G(\mathcal{X})-\delta_{G}+\mathcal{Y};
      G'(\mathcal{X})\Delta_{\mathcal{X}}-\Delta_{\delta_{G}}+\Delta_{\mathcal{Y}}\big)\\
   \end{matrix}\right),
  \end{align}
  where $[h'(\overline{\mathcal{X}})]^*$ and $[G'(\overline{\mathcal{X}})]^*$ denote the adjoint
  of $h'(\overline{\mathcal{X}})$ and $G'(\overline{\mathcal{X}})$, respectively. Let
  \[
   \mathscr{F}\!:=\big\{(\Delta_{\mathcal{X}},\Delta_{\lambda},\Delta_{\mathcal{Y}})\in\mathbb{X}\times\mathbb{R}^p\times\mathbb{X}
              \ |\ \widetilde{\Psi}'\big((0,\overline{\mathcal{X}},\overline{\lambda},\overline{\mathcal{Y}});
               (0,\Delta_{\mathcal{X}},\Delta_{\lambda},\Delta_{\mathcal{Y}})\big)\!=0\big\}.
 \]
 By Lemma \ref{Graph-dir} and the directional differentiability of $\widetilde{\Psi}$ and $\Psi$,
 it follows that
 \begin{equation}\label{inclusion}
   D\widetilde{\mathcal{J}}\big((0,0,0)|(\overline{\mathcal{X}},\overline{\lambda},\overline{\mathcal{Y}})\big)(0,0,0)
   \subseteq\mathscr{F}=D\mathcal{J}\big((0,0,0)|(\overline{\mathcal{X}},\overline{\lambda},\overline{\mathcal{Y}})\big)(0,0,0).
 \end{equation}
  Next we prove that $\mathscr{F}=\{(0,0,0)\}$. Clearly, $(0,0,0)\in\mathscr{F}$.
  Suppose on the contradiction that there exists a nonzero $\Delta_{\mathscr{F}}=(\Delta_{\mathcal{X}},\Delta_{\lambda},\Delta_{\mathcal{Y}})\in\mathscr{F}$.
  Together with \eqref{Fdir}, we have that
 \begin{equation}\label{Natural-equa1}
   \nabla_{\mathcal{X}\mathcal{X}}^2L(\overline{\mathcal{X}},\overline{\lambda},\overline{\mathcal{Y}})\Delta_{\mathcal{X}}
   +[h'(\overline{\mathcal{X}})]^*\Delta_{\lambda}+[G'(\overline{\mathcal{X}})]^*\Delta_{\mathcal{Y}}=0,
 \end{equation}
 \begin{equation}\label{Natural-equa2}
   h'(\overline{\mathcal{X}})\Delta_{\mathcal{X}}=0,
 \end{equation}
 \begin{equation}\label{Natural-equa3}
  \Pi_{K}'\big(G(\overline{\mathcal{X}})+\overline{\mathcal{Y}};G'(\overline{\mathcal{X}})\Delta_{\mathcal{X}}+\Delta_{\mathcal{Y}}\big)
   =G'(\overline{\mathcal{X}})\Delta_{\mathcal{X}}.
  \end{equation}
 Making the inner product with $\Delta_{\mathcal{X}}$ for the two sides of \eqref{Natural-equa1}
 and using \eqref{Natural-equa2} yields that
 \begin{equation}\label{temp-SOSC-equa}
   \left\langle \Delta_{\mathcal{X}},
   \nabla_{\mathcal{X}\mathcal{X}}^2L(\overline{\mathcal{X}},\overline{\lambda},\overline{\mathcal{Y}})\Delta_{\mathcal{X}}\right\rangle
   +\left\langle G'(\overline{\mathcal{X}})\Delta_{\mathcal{X}},\Delta_{\mathcal{Y}}\right\rangle =0.
 \end{equation}
  From equation \eqref{Natural-equa3} and Lemma \ref{tcone-lemma}, it immediately follows that
  \begin{equation}\label{Natural-equa4}
   \big(\Delta_{\mathcal{Y}},G'(\overline{\mathcal{X}})\Delta_{\mathcal{X}}\big)
   \in {\rm gph}\,D\mathcal{N}_{K^{\circ}}\big(\overline{\mathcal{Y}}|G(\overline{\mathcal{X}})\big).
 \end{equation}
  By \eqref{equa3-main-prop} of Proposition \ref{main-prop},
  \(
   \langle \Delta_{\mathcal{Y}},G'(\overline{\mathcal{X}})\Delta_{\mathcal{X}}\rangle\!= -\Upsilon_{G(\overline{\mathcal{X}})}\big(\overline{\mathcal{Y}},G'(\overline{\mathcal{X}})\Delta_{\mathcal{X}}\big).
 \)
 Then, we have
 \begin{equation}\label{temp-main-equa1}
   \langle \Delta_{\mathcal{X}},\nabla_{\mathcal{X}\mathcal{X}}^2L(\overline{\mathcal{X}},
   \overline{\lambda},\overline{\mathcal{Y}})\Delta_{\mathcal{X}}\rangle
   -\Upsilon_{G(\overline{\mathcal{X}})}\big(\overline{\mathcal{Y}},G'(\overline{\mathcal{X}})\Delta_{\mathcal{X}}\big)=0.
 \end{equation}
  In addition, from \eqref{Natural-equa4} and \eqref{equa1-main-prop} of Proposition \ref{main-prop},
 \(
  G'(\overline{\mathcal{X}})\Delta_{\mathcal{X}}\in \mathcal{T}_{K}(G(\overline{\mathcal{X}}))\cap \overline{\mathcal{Y}}^\perp,
 \)
 which along with \eqref{Natural-equa1} and $L_{\mathcal{X}}'(\overline{\mathcal{X}};\overline{\lambda},\overline{\mathcal{Y}})=0$
 implies that $\Delta_{\mathcal{X}}\in \mathcal{C}(\overline{\mathcal{X}})$.
 So, by \eqref{temp-main-equa1} and the second-order sufficient condition,
 we have $\Delta_{\mathcal{X}}=0$. Substituting $\Delta_{\mathcal{X}}=0$ into \eqref{Natural-equa1} yields that
 \(
   [h'(\overline{\mathcal{X}})]^*\Delta_{\lambda}+[G'(\overline{\mathcal{X}})]^*\Delta_{\mathcal{Y}}=0.
 \)
 Since $\Delta_{\mathcal{X}}\!=0$, equations \eqref{Natural-equa4} and \eqref{equa2-main-prop} imply that
 $\Delta_{\mathcal{Y}}\in [\mathcal{T}_{K}(G(\overline{\mathcal{X}}))\cap \overline{\mathcal{Y}}^\perp]^{\circ}$.
 Then,
 \(
  (\Delta_{\lambda},\Delta_{\mathcal{Y}})\in
   \Big[\left(\begin{matrix}
       h'(\overline{\mathcal{X}})\\
       G'(\overline{\mathcal{X}})
       \end{matrix}\right)\mathbb{X}\Big]^{\perp}
   \bigcap\Big[\begin{matrix}
       0\\ \mathcal{T}_{K}(G(\overline{\mathcal{X}}))\cap\overline{\mathcal{Y}}^{\perp}
       \end{matrix}\Big]^{\circ}.
 \)
 Thus, the strict Robinson's CQ at $(\overline{\lambda},\overline{\mathcal{Y}})$ implies that
 $(\Delta_{\lambda},\Delta_{\mathcal{Y}})=(0,0)$.
 Consequently, $(\Delta_{\mathcal{X}},\Delta_{\lambda},\Delta_{\mathcal{Y}})=0$.
 This yields a contradiction. Hence,
 \(
   D\mathcal{J}\big((0,0,0)|(\overline{\mathcal{X}},\overline{\lambda},\overline{\mathcal{Y}})\big)(0,0,0)
   =\mathscr{F}=\{(0,0,0)\},
 \)
 and then
 $\{(0,0,0)\}\subseteq D\widetilde{\mathcal{J}}\big((0,0,0)|(\overline{\mathcal{X}},\overline{\lambda},\overline{\mathcal{Y}})\big)(0,0,0)
 \subseteq\{(0,0,0)\}$. By Lemma \ref{SN-cond}, $\mathcal{J}$ and
 $\widetilde{\mathcal{J}}$ are both locally upper Lipschitz at the origin for $(\overline{\mathcal{X}},\overline{\lambda},\overline{\mathcal{Y}})$.

 \medskip
 \noindent
 {\bf (b)} Suppose that the strict Robinson's CQ does not hold at $(\overline{\lambda},\overline{\mathcal{Y}})$.
 From \eqref{SCQ}, we have
 \begin{equation}
    \left[\left(\begin{matrix}
       h'(\overline{\mathcal{X}})\\
       G'(\overline{\mathcal{X}})
       \end{matrix}\right)\mathbb{X}
       +\left(\begin{matrix}
       0\\ \mathcal{T}_{K}(G(\overline{\mathcal{X}}))\cap\overline{\mathcal{Y}}^{\perp}
       \end{matrix}\right)\right]^{\circ}\neq
       \left(\begin{matrix}
       0\\ 0
       \end{matrix}\right),
  \end{equation}
  which is equivalent to saying that there exists $0\neq (\widehat{\lambda},\widehat{\mathcal{Y}})\in \mathbb{R}^p\times \mathbb{X}$
  such that
  \begin{equation*}
    (\widehat{\lambda},\widehat{\mathcal{Y}})\in\left[\left(\begin{matrix}
       h'(\overline{\mathcal{X}})\\
       G'(\overline{\mathcal{X}})
       \end{matrix}\right)\mathbb{X}\right]^{\bot}\bigcap\left(\begin{matrix}
       \mathbb{R}^p\\ [\mathcal{T}_{K}(G(\overline{\mathcal{X}}))\cap\overline{\mathcal{\mathcal{Y}}}^{\perp}]^{\circ}
       \end{matrix}\right).
  \end{equation*}
  Consequently, we have that
  $[h'(\overline{\mathcal{X}})]^*\widehat{\lambda}+[G'(\overline{\mathcal{X}})]^* \widehat{\mathcal{Y}}=0$
  and the following relation holds:
  \begin{equation}\label{equa2-partb}
  \widehat{\mathcal{Y}}\in \big(\mathcal{T}_{K}(G(\overline{\mathcal{X}}))\cap\overline{\mathcal{\mathcal{Y}}}^{\perp}\big)^{\circ}
  ={\rm cl}\big(\mathcal{N}_{K}(G(\overline{\mathcal{X}}))+[\![\overline{\mathcal{Y}}]\!]\big)
  =\mathcal{T}_{\mathcal{N}_K(G(\overline{\mathcal{X}}))}(\overline{\mathcal{Y}})
  \end{equation}
  where the last equality is due to \cite[Example 2.62]{BS00}.
  By the definition of contingent cone, it is easy to verify that
  if $\mathcal{Z}\in\mathcal{T}_{\mathcal{N}_K(G(\overline{\mathcal{X}}))}(\overline{\mathcal{Y}})$,
  then $(0,\mathcal{Z})\in\mathcal{T}_{{\rm gph}\,\mathcal{N}_K}(G(\overline{\mathcal{X}}),\overline{\mathcal{Y}})$.
  Thus, equation \eqref{equa2-partb} means that
  $(0,\widehat{\mathcal{Y}})\in\mathcal{T}_{{\rm gph}\,\mathcal{N}_K}(G(\overline{\mathcal{X}}),\overline{\mathcal{Y}})$.
  Using Lemma \ref{tcone-lemma}, we have that
  \(
    \Pi_{K}'\big(G(\overline{\mathcal{X}})+\overline{\mathcal{Y}};\widehat{\mathcal{Y}}\big)=0.
  \)
  Together with $[h'(\overline{\mathcal{X}})]^*\widehat{\lambda}+[G'(\overline{\mathcal{X}})]^* \widehat{\mathcal{Y}}=0$,
  we obtain that
  \[
   \Psi'\big((\overline{\mathcal{X}},\overline{\lambda}, \overline{\mathcal{Y}});(0,\widehat{\lambda},\widehat{\mathcal{Y}})\big)=0.
  \]
   By Lemma \ref{Graph-dir}, we have $(0,\widehat{\lambda},\widehat{\mathcal{Y}})\in
  D\mathcal{J}((0,0,0)|(\overline{\mathcal{X}},\overline{\lambda}, \overline{\mathcal{Y}}))(0,0,0)
  =\{(0,0,0)\}$, where the equality is due to the locally upper Lipschitz $\mathcal{J}$ at the origin
  for $(\overline{\mathcal{X}},\overline{\lambda},\overline{\mathcal{Y}})$ and Lemma \ref{SN-cond}.
  Thus, we get $(\widehat{\lambda}, \widehat{\mathcal{Y}})=0$,
  a contradiction to $(\widehat{\lambda}, \widehat{\mathcal{Y}})\ne 0$.
  The proof is completed.
 \end{proof}
%------------------------------------------------------------------------------------------------Remark
 \begin{remark}\label{main-remark}
  {\bf(a)} From the inclusion relation in \eqref{inclusion} and Lemma \ref{SN-cond},
  it is not hard to see that the locally upper Lipschitz of $\mathcal{J}$ at the origin
  for $(\overline{\mathcal{X}},\overline{\lambda},\overline{\mathcal{Y}})$
  implies that of $\widetilde{\mathcal{J}}$ at the origin for
  $(\overline{\mathcal{X}},\overline{\lambda},\overline{\mathcal{Y}})$.
  Hence, when the condition of part (b) is replaced by the locally upper Lipschitz of
  $\widetilde{\mathcal{J}}$ at the origin for $(\overline{\mathcal{X}},\overline{\lambda},\overline{\mathcal{Y}})$,
  the strict Robinson's CQ may not hold.

  \medskip
  \noindent
  {\bf(b)} Let $\Omega$ be the KKT point set of \eqref{NMCP}.
  By Lemma \ref{equiv-chara-LUL}, the locally upper Lipschitz of $\mathcal{J}$ at $0$ for
  $\overline{\mathcal{W}}=\!(\overline{\mathcal{X}},\overline{\lambda},\overline{\mathcal{Y}})$
  means that there are a constant $\vartheta\!\ge 0$ and a small $\varepsilon\!>0$ such that
  \[
    {\rm dist}\big((\mathcal{X},\lambda,\mathcal{Y}),\Omega)
    \le \|(\mathcal{X},\lambda,\mathcal{Y})-(\overline{\mathcal{X}},\overline{\lambda},\overline{\mathcal{Y}})\|
    \le \vartheta\|\Psi(\mathcal{X},\lambda,\mathcal{Y})\|\quad \forall(\mathcal{X},\lambda,\mathcal{Y})\in\mathbb{B}(\overline{\mathcal{W}},\varepsilon),
  \]
  which provides a local error bound for estimating the distance from any points
  in the neighborhood of $\overline{\mathcal{W}}$ to $\Omega$, and the bound is only
  related to the KKT residual $\|\Psi(\mathcal{X},\lambda,\mathcal{Y})\|$.

  \medskip
  \noindent
  {\bf(c)} By the definitions of $\widetilde{\Psi}$ and $\Psi$,
  the multifunction $\widetilde{J}$ can be equivalently written as
  \begin{equation}\label{WJmap-equiv}
  \widetilde{\mathcal{J}}(\delta)=\big\{(\mathcal{X},\lambda,\mathcal{Y})\in \mathbb{X}\times\mathbb{R}^{p}\times\mathbb{X}
  \ |\ \Psi\big(\mathcal{X},\lambda,\mathcal{Y}\big)\in \delta+\mathcal{N}_{\mathbb{X}\times\mathbb{R}^p\times K^{\circ}}(\mathcal{X},\lambda,\mathcal{Y})\big\}.
 \end{equation}
  Thus, by equation \eqref{WJmap-equiv} and Lemma \ref{equiv-chara-LUL},
  the locally upper Lipschitz of $\widetilde{\mathcal{J}}$ at the origin for
  $\overline{\mathcal{W}}=(\overline{\mathcal{X}},\overline{\lambda},\overline{\mathcal{Y}})$
  implies that there exist a constant $\vartheta\ge 0$ and a small $\varepsilon>0$ such that
  \[
    {\rm dist}\big((\mathcal{X},\lambda,\mathcal{Y}),\Omega)
    \le \vartheta(\|\Psi(\mathcal{X},\lambda,\mathcal{Y})\|+\|\Pi_{K}(\mathcal{Y})\|)
    \quad\forall (\mathcal{X},\lambda,\mathcal{Y})\in\mathbb{B}(\overline{\mathcal{W}},\varepsilon).
  \]
%  and then there exist constants $\theta'>0$ and $c>0$ such that for any $(\mathcal{X},\lambda,\mathcal{Y})\in\mathbb{B}(\overline{\mathcal{W}},\varepsilon)$,
%  \begin{align*}
%    {\rm dist}\big((\mathcal{X},\lambda,\mathcal{Y}),\Omega)
%    &\le \vartheta\|\Psi^{+}(\mathcal{X},\lambda,\mathcal{Y})\|+\|\Pi_{K}(\mathcal{Y})\|_F\\
%    &\le \vartheta'\|\Psi(\mathcal{X},\lambda,\mathcal{Y})\| + c(\|\Pi_{K^{\circ}}(\mathcal{Y})\|_F+\|\Pi_{K}(\mathcal{Y})\|_F).
%  \end{align*}
  \end{remark}

  \bigskip
  \noindent
  {\bf\large Appendix A}
  \noindent
%-------------------------------------------------------------------------------------------------------lemma
 \begin{alemma}\label{equiv-LUL}
  Let $(\overline{y},\overline{z})\in{\rm gph}\,\mathcal{S}$ for a multifunction
  $\mathcal{S}\!:\mathbb{Y}\rightrightarrows \mathbb{Z}$. Then,
  the following two statements are equivalent:
  \begin{description}
   \item[(a)] there exist a constant $\mu\ge 0$ and neighborhoods
              $\mathcal{U}$ of $\overline{y}$ and $\mathcal{V}$ of $\overline{z}$ such that
              \[
                \mathcal{V}\cap\mathcal{S}(\overline{y})=\{\overline{z}\}\ \ {\rm and}\ \
                \mathcal{S}(y)\cap \mathcal{V} \subset \{\overline{z}\} + \mu\|y-\overline{y}\|\mathbb{B}_{\mathbb{Y}}
                \ \ {\rm for\ all}\ y\in \mathcal{U};
              \]

   \item[(b)] there exist a constant $\mu'\ge 0$ and a neighborhood $\mathcal{V}$ of $\overline{z}$ such that
              \[
                \mathcal{V}\cap\mathcal{S}(\overline{y})=\{\overline{z}\}\ \ {\rm and}\ \
                \mathcal{S}(y)\cap \mathcal{V} \subset \{\overline{z}\} + \mu'\|y-\overline{y}\|\mathbb{B}_{\mathbb{Y}}
                \ \ {\rm for\ all}\ y\in \mathbb{Y}.
              \]
  \end{description}
 \end{alemma}
 \begin{proof}
  Clearly, (b) implies (a). It suffices to argue that (a) implies (b). For this purpose,
  we assume that (a) holds for neighborhoods $\mathcal{U}=\mathbb{B}(\overline{y},\delta)$ and
  $\mathcal{V}=\mathbb{B}(\overline{z},\varepsilon)$ with $\delta>0$ and $\varepsilon>0$.
  We proceed the arguments by the two cases as shown below.

  \medskip
  \noindent
  {\bf Case 1: $\varepsilon>\delta$.} We show that (b) holds for
  $\mathcal{V}'=\mathbb{B}(\overline{z},\varepsilon')$ with $\varepsilon'=\frac{\delta}{2}$.
  Since $\varepsilon'<\varepsilon$, it is clear that $\mathcal{V}'\cap\mathcal{S}(\overline{y})=\{\overline{z}\}$.
  Moreover, since part (a) holds, it immediately follows that
  \[
    \mathcal{S}(y)\cap \mathbb{B}(\overline{z},\varepsilon')\subset \{\overline{z}\}
     + \mu\|y-\overline{y}\|\mathbb{B}_{\mathbb{Y}}\ \ {\rm when}\
    y\in\mathbb{B}(\overline{y},\delta).
  \]
  In addition, by noting that $\mathbb{B}(\overline{z},\varepsilon')\subset \{\overline{z}\}+\frac{3}{4}\delta\mathbb{B}_{\mathbb{Y}}$,
  we also have
  \[
     \mathcal{S}(y)\cap \mathbb{B}(\overline{z},\varepsilon')\subset\mathbb{B}(\overline{z},\varepsilon')
     \subset \{\overline{z}\}+\|y-\overline{y}\|\mathbb{B}_{\mathbb{Y}}
     \ \ {\rm when}\ y\notin\mathbb{B}(\overline{y},\delta)
  \]
  Thus, part (b) directly follows from the last two equations with $\mu'=\max(\mu,1)$.

   \medskip
  \noindent
  {\bf Case 2: $\varepsilon\le\delta$.} Using the same arguments as for Case 1 can verify that
  (b) holds for $\mathcal{V}'=\mathbb{B}(\overline{z},\varepsilon')$ with $\varepsilon'={\varepsilon}/{2}$.
  Here, we omit the details for simplicity.
 \end{proof}

  \bigskip

  \medskip
  \noindent
  {\bf\large Appendix B}

  \medskip
  \noindent
  This part includes some lemmas used for the proof of Proposition \ref{main-prop}
  and the proof of Proposition \ref{main-prop}, which requires the mappings
  $\mathcal{G}\!:\mathbb{R}^{q\times q}\to\mathbb{S}^{q}$
  and $\mathcal{H}\!:\mathbb{R}^{q\times q}\to\mathbb{R}^{q\times q}$ as
  \begin{equation}\label{GH-operator}
    \mathcal{G}(Z):=(Z+Z^{\mathbb{T}})/2\ \ {\rm and}\ \
    \mathcal{H}(Z):=(Z-Z^{\mathbb{T}})/2\ \ {\rm for}\ \ Z\in\mathbb{R}^{q\times q},
  \end{equation}
  and the matrices $\mathcal{E}_1,\mathcal{E}_2\in\!\mathbb{R}^{m\times m}$ and
  $\mathcal{F}\!\in\mathbb{R}^{m\times(n-m)}$ associated to the given
  $\overline{X},X\in\!\mathbb{R}^{m\times n}$:
  \begin{align}\label{ME1}
   \big(\mathcal{E}_1\big)_{ij}
   :=\!\left\{\begin{array}{cl}
       \!\frac{\sigma_i(\overline{X})-\sigma_j(\overline{X})}{\sigma_i(X)-\sigma_j(X)} & {\rm if}\ \sigma_i(X)\ne \sigma_j(X),\\
        0 & {\rm otherwise},
     \end{array}\right.\ {\rm for}\ i,j\in\{1,\ldots,m\},\quad\\
    \big(\mathcal{E}_2\big)_{ij}
   :=\left\{\begin{array}{cl}
       \frac{\sigma_i(\overline{X})+\sigma_j(\overline{X})}{\sigma_i(X)+\sigma_j(X)} & {\rm if}\ \sigma_i(X)\!+\!\sigma_j(X)\ne 0,\\
        0 & {\rm otherwise},
     \end{array}\right.\ \ {\rm for}\ i,j\in\{1,\ldots,m\},\\
     \label{ME2}
   \big(\mathcal{F}\big)_{ij}
   :=\left\{\begin{array}{cl}
       \frac{\sigma_i(\overline{X})}{\sigma_i(X)} & {\rm if}\ \sigma_i(X)\ne 0,\\
        0 & {\rm otherwise},
     \end{array}\right.\ {\rm for}\ i\in\{1,\ldots,m\},j\in\{1,\ldots,n\!-\!m\}.
  \end{align}
  Unless otherwise stated, in the sequel, when $X$ has the SVD as in \eqref{X-SVD},
  we write
  \[
    \widetilde{Z}=\overline{U}^{\mathbb{T}}Z\overline{V}\ \ {\rm and}\ \
    \widetilde{Z}_1=\overline{U}^{\mathbb{T}}Z\overline{V}_1\quad{\rm for\ any}\ Z\in\mathbb{R}^{m\times n}.
  \]

  Firstly, we recall from \cite[Lemma 3.15]{DingC10} a result on the projection of $(t,X)$ onto $K$.
%-------------------------------------------------------------------------------------------
  \begin{alemma}\label{k0k1-propery}
   Let $(t,X)\notin{\rm int}\,K\cup{\rm int}\,K^{\circ}$ be given.
   Write $(\overline{t},\overline{X})=\Pi_K(t,X)$ and $\overline{\sigma}=\sigma(\overline{X})$.
   \begin{enumerate}
     \item[(i)] If $\overline{\sigma}_k>0$, then there exist $\theta>0$ and $\overline{u}\in\mathbb{R}_{+}^m$
               such that $\overline{\sigma}=\sigma(X)-\theta\overline{u}$ with
               $\overline{u}_i=1$ for $i=1,\ldots,k_0$,
               $1\ge\overline{u}_{k_0+1}\ge\ldots\ge\overline{u}_{k_1}\ge 0$
               with $\sum_{i=k_0+1}^{k_1}\overline{u}_{i}={k}-k_0$,
               and $\overline{u}_i=0$ for $i=k_1\!+\!1,\ldots,m$, where $k_0\in[0,k\!-\!1]$ and
               $k_1\in[k,m]$ are integers such that
              \begin{equation}\label{sigmabar-case1}
                \overline{\sigma}_1\ge\cdots\ge\overline{\sigma}_{k_0}\!>\overline{\sigma}_{k_0+1}=\!\cdots=\!\overline{\sigma}_{k}
                =\cdots=\overline{\sigma}_{k_1}:=\overline{\nu}>\overline{\sigma}_{k_1+1}\!\ge\cdots\ge\overline{\sigma}_m\!\ge 0.
              \end{equation}
              In the subsequent discussion, for this case we always write
              \begin{equation}\label{abg-case1}
               \alpha\!:=\!\{1,\ldots,k_0\},\,\beta\!:=\!\{k_0\!+\!1,\ldots,k_1\},\,
               \gamma\!:=\!\{k_1\!+\!1,\ldots,m\},\overline{\gamma}=\{1,\ldots,m\}\backslash \gamma.
              \end{equation}

    \item[(ii)] If $\overline{\sigma}_k=0$, then there exist $\theta>0$ and $\overline{u}\in\mathbb{R}_{+}^m$ such that
                $\overline{\sigma}=\sigma(X)-\theta\overline{u}$ with $\overline{u}_i=1$ for $i=1,\ldots,k_0$
                and $1\ge\overline{u}_{k_0+1}\ge \overline{u}_{k_0+2}\ge\ldots\ge\overline{u}_{m}\ge 0$ with
                $\sum_{i=k_0+1}^{m}\overline{u}_{i}\le k-k_0$, where $0\le k_0\le k\!-\!1$ is an integer such that
               \begin{equation}\label{sigmabar-case2}
                 \overline{\sigma}_1\ge\ldots\ge\overline{\sigma}_{k_0}>\overline{\sigma}_{k_0+1}=\ldots=\overline{\sigma}_k=\ldots=\overline{\sigma}_{m}=0.
               \end{equation}
               In the subsequent discussion, for this case we always write
              \begin{equation}\label{abg-case2}
                \alpha\!:=\{1,2,\ldots,k_0\}\ \ {\rm and}\ \ \beta\!:=\{k_0\!+\!1,k_0\!+\!2,\ldots,m\}.
              \end{equation}
   \end{enumerate}
  Also, for the two cases we sometimes use the partition for $\beta:=\beta_1\cup\beta_2\cup\beta_3$ with
  \[
  \beta_1:=\{i\in\beta\ |\ \overline{u}_i=1\},\,\beta_2:=\{i\in\beta\ |\ \overline{u}_i\in(0,1)\}
  \ {\rm and}\ \beta_3:=\{i\in\beta\ |\ \overline{u}_i=0\}.
  \]
  \end{alemma}

  Next we recall from \cite[Proposition 3.16]{DingC10} the expression of
  the directional derivative of the projection operator $\Pi_K$ at
  $(t,X)\notin{\rm int}\,K\cup{\rm int}\,K^{\circ}$, which is stated as follows.
%---------------------------------------------------------------------------------------
  \begin{alemma}\label{Lemma1-pdird}
   Let $(t,X)\notin{\rm int}\,K\cup{\rm int}\,K^{\circ}$ be given.
   Write $(\overline{t},\overline{X})=\Pi_K(t,X)$ and $\overline{\sigma}=\sigma(\overline{X})$.
   Suppose that $X$ has the SVD as in \eqref{X-SVD} with $a,b,c$ and
   $a_l\ (l=1,2,\ldots,r)$ given by \eqref{abc} and \eqref{al}.
   Then, the directional derivative of $\Pi_K$ at $(t,X)$ along
   $(\tau,H)\in\mathbb{R}\times\mathbb{R}^{m\times n}$ is
   \[
    \Pi_{K}'((t,X);(\tau,H))=\big(\Phi_0(\tau,\bm D(\widetilde{H})),
    \overline{U}\Xi(\tau,\bm D(\widetilde{H}))\overline{V}^{\mathbb{T}}\big)
  \]
  where $\Phi_0(\tau,\bm D(\widetilde{H}))$ and $\Xi(\tau,\bm D(\widetilde{H}))$ will be stated by
   $\overline{\sigma}_k>0$ and $\overline{\sigma}_k=0$, respectively.

  \medskip
  \noindent
  {\bf Case 1:} $\overline{\sigma}_k>0$. Let $r_0,r_1\!\in\{1,\ldots,r\}$ be such that
  $\alpha=\bigcup_{l=1}^{r_0}a_l,\beta=\bigcup_{l=r_0+1}^{r_1}a_l$ and $\gamma=\bigcup_{l=r_1+1}^{r+1}a_l$,
  where $\alpha,\beta,\gamma$ and $\overline{\gamma}$ are defined by \eqref{abg-case1} with
  integers $k_0\in[0,k\!-\!1]$ and $k_1\in[k,m]$ such that equation \eqref{sigmabar-case1} holds.
  Write $\mathbb{W}:=\mathbb{R}\times\mathbb{S}^{|a_1|}\times\mathbb{S}^{|a_2|}\times\cdots\times\mathbb{S}^{|a_{r_1}|}$.
  For any $(\zeta,W)\in\mathbb{W}$ with $W=[W_1\ W_2\ \cdots\ W_{r_1}]\in \mathbb{S}^{|a_1|}\times\mathbb{S}^{|a_2|}\times\cdots\times\mathbb{S}^{|a_{r_1}|}$,
  we define
  \[
    \Phi_0(\zeta,W):=\!\phi_0(\zeta,{\bm\kappa}(W))\in\mathbb{R}\ \ {\rm and}\ \
    \Phi_l(\zeta,W):=\!R_l{\rm Diag}(\phi_l(\zeta,{\bm\kappa}(W)))R_l^{\mathbb{T}}\in\mathbb{S}^{|a_l|}
  \]
  with $R_l\in\mathbb{O}^{|a_l|}(W_l)$ for $l=1,\ldots,r_1$, where
  $(\phi_0(\zeta,\bm{\kappa}(W)),\phi_1(\zeta,\bm{\kappa}(W)),\ldots,\phi_{r_1}(\zeta,\bm{\kappa}(W)))$
  with $\bm{\kappa}(W):=(\lambda(W_1),\lambda(W_2),\ldots,\lambda(W_{r_1}))\in\mathbb{R}^{k_1}$
  is the unique solution of the problem
  \begin{align*}
   &\min\ \frac{1}{2}\big[(\eta-\zeta)^2+\|d-\bm{\kappa}(W)\|^2\big]\\
   &\ {\rm s.t.}\ \langle e_{\alpha},d_{\alpha}\rangle+s_{(k-k_0)}(d_{\beta})\le \eta
  \end{align*}
  if $(t,X)\in{\rm bd}K$, and otherwise is the unique optimal solution of the problem
  \begin{align*}
   &\min\ \frac{1}{2}\big[(\eta-\zeta)^2+\|d-\bm{\kappa}(W)\|^2\big]\\
   &\ {\rm s.t.}\ \langle e_{\alpha},d_{\alpha}\rangle+s_{(k-k_0)}(d_{\beta})\le \eta\\
   &\qquad \langle e_{\alpha},d_{\alpha}\rangle+\langle\overline{u}_{\beta},d_{\beta}\rangle=\eta.
  \end{align*}
  Here, $s_{(k-k_0)}\!:\mathbb{R}^{|\beta|}\to\mathbb{R}$ is defined by
  \(
    s_{(k-k_0)}(z):={\textstyle\sum_{i=1}^{k-k_0}}z_{i}^{\downarrow}
  \)
  for $z\in\mathbb{R}^{|\beta|}$. Then,
  \begin{equation}\label{Xi-Case1}
   \Xi(\tau,\bm D(\widetilde{H}))
   \!:=\!\bm T(\widetilde{H})
   +\!\left[\begin{matrix}
    \Phi_1\big(\tau,\bm D(\widetilde{H})\big) & 0 & 0 & 0 & 0\\
    0 & \ddots & 0 & 0 & 0\\
    0 & 0 &\Phi_{r_1}\big(\tau,\bm D(\widetilde{H})\big)& 0 &0\\
    0 & 0 &0 &\mathcal{G}(\widetilde{H}_{\gamma\gamma})
    & \left(\begin{matrix}
        0\\ \widetilde{H}_{bc}
      \end{matrix}\right)\\
    \end{matrix}\right]
  \end{equation}
  where $\bm D(Z)$ and $\bm T(Z)$ for any $Z=[Z_1\ \ Z_2]\!\in\mathbb{R}^{m\times n}$
  with $Z_1\in\mathbb{R}^{m\times m}$ are defined by
  \begin{equation}\label{MD-Case1}
   \bm D(Z):=\big(\mathcal{G}(Z_{a_1a_1}),\mathcal{G}(Z_{a_2a_2}),\ldots,\mathcal{G}(Z_{a_{r_1}a_{r_1}})\big)
  \end{equation}
  and
  \begin{equation}\label{MTZ-Case1}
   \bm T(Z):=\big[\mathcal{E}_1\circ\mathcal{G}(Z_1)+\mathcal{E}_2\circ\mathcal{H}(Z_1)\ \
             \mathcal{F}_{\overline{\gamma}\cup\gamma,c}\circ Z_{\overline{\gamma}\cup\gamma,c}\big].
  \end{equation}
  Here ``$\circ$'' means the Hardmard product. In particular, from \cite[Page 127]{DingC10}, we know that
  \begin{equation}\label{Proj-C1set}
   \Pi_{\mathcal{C}_1}(\tau,\bm D(\widetilde{H}))
   =\big(\Phi_0(\tau,\bm D(\widetilde{H})),\Phi_1(\tau,\bm D(\widetilde{H})),\ldots,
    \Phi_{r_1}(\tau,\bm D(\widetilde{H}))\big),
  \end{equation}
  where $\mathcal{C}_1\subseteq\mathbb{W}$ is a closed convex cone which, if $(t,X)\in{\rm bd}\,K$,
  has the form
  \begin{equation*}
   \mathcal{C}_1=\Big\{(\zeta,W)\in\mathbb{W}\ |\ {\textstyle\sum_{l=1}^{r_0}}{\rm tr}(W_l)+s_{(k-k_0)}(\bm{\kappa}_{\beta}(W))\le\zeta\Big\},
  \end{equation*}
  and otherwise takes the following form
 \begin{equation}\label{C1-cone}
  \mathcal{C}_1=\Big\{(\zeta,W)\!\in\!\mathbb{W}\ | \sum_{l=1}^{r_0}{\rm tr}(W_l)\!+\!s_{(k-k_0)}(\bm{\kappa}_{\beta}(W))\!\le\zeta,
  \sum_{l=1}^{r_0}{\rm tr}(W_l)\!+\!\langle \overline{u}_\beta,\bm{\kappa}_{\beta}(W)\rangle\!=\!\zeta\Big\}.
 \end{equation}

  \noindent
  {\bf Case 2:} $\overline{\sigma}_k=0$. Let $r_0\in\{1,2,\ldots,r\}$ be such that
  $\alpha=\bigcup_{l=1}^{r_0}a_l$ and $\beta=\bigcup_{l=r_0+1}^{r+1}a_l$,
  where $\alpha$ and $\beta$ are defined by \eqref{abg-case2} with integer $k_0\in[0,k-1]$
  such that \eqref{sigmabar-case2} holds.  Let
  \(
   \mathbb{W}:=\mathbb{R}\times\mathbb{S}^{|a_1|}\times\cdots\times\mathbb{S}^{|a_{r}|}\times\mathbb{R}^{|b|\times(|b|+|c|)}.
  \)
  For any $(\zeta,W)$ with $W=[W_1\ \cdots\ W_{r}\ \ W_{r+1}]$ $ \in\mathbb{S}^{|a_1|}\times\cdots\times\mathbb{S}^{|a_{r}|}\times\mathbb{R}^{|b|\times(|b|+|c|)}$,
  we define $\Phi_0(\zeta,W):=\phi_0(\zeta,{\bm\kappa}(W))\in\mathbb{R}$,
  $\Phi_l(\zeta,W):=R_l{\rm Diag}(\phi_l(\zeta,\bm{\kappa}(W)))R_l^{\mathbb{T}}\in\mathbb{S}^{|a_l|}$
  with $R_l\in\mathbb{O}^{|a_l|}(W_l)$ for $l=1,2,\ldots,r$, and
  $\Phi_{r+1}(\zeta,W):=U[{\rm Diag}(\phi_{r+1}(\zeta,\bm{\kappa}(W)))\ \ 0]V^{\mathbb{T}}\in\mathbb{R}^{|b|\times(|b|+|c|)}$
  with $(U,V)\in\mathbb{O}^{|b|,(|b|+|c|)}(W_{r+1})$, where
  \(
   (\phi_0(\zeta,\bm{\kappa}(W)),\phi_1(\zeta,\bm{\kappa}(W)),\ldots,\phi_{r+1}(\zeta,\bm{\kappa}(W)))
  \)
  with $\bm{\kappa}(W)\!=(\lambda(W_1),\ldots,\lambda(W_{r}),\sigma(W_{r+1}))$
  is the unique optimal solution of the following convex minimization problem
  \begin{align*}
   &\min\ \frac{1}{2}\big[(\eta-\zeta)^2+\|d-\bm{\kappa}(W)\|^2\big]\\
   &\ {\rm s.t.}\ \langle e_{\alpha},d_{\alpha}\rangle+\|d_{\beta}\|_{(k-k_0)}\le \eta
  \end{align*}
  if $(t,X)\in{\rm bd}K$, and otherwise is the unique optimal solution of the convex problem
  \begin{align*}
   &\min\ \frac{1}{2}\big[(\eta-\zeta)^2+\|d-\bm{\kappa}(W)\|^2\big]\\
   &\ {\rm s.t.}\ \langle e_{\alpha},d_{\alpha}\rangle+\|d_{\beta}\|_{(k-k_0)}\le \eta\\
   &\qquad \langle e_{\alpha},d_{\alpha}\rangle+\langle\overline{u}_{\beta},d_{\beta}\rangle=\eta.
  \end{align*}
  Here, $\|\cdot\|_{(k-k_0)}\!:\mathbb{R}^{|\beta|}\to\mathbb{R}$ is defined by
  $\|z\|_{(k-k_0)}:={\textstyle\sum_{i=1}^{k-k_0}}|z|_{i}^{\downarrow}$ for $z\in\mathbb{R}^{|\beta|}$.
  Then,
  \begin{equation}\label{Xi-Case2}
   \Xi(\tau,\bm D(\widetilde{H}))
   :=\bm T(\widetilde{H})
   +\left[\begin{matrix}
    \Phi_1\big(\tau,\bm D(\widetilde{H})\big) & 0 & 0 & 0\\
    0 & \ddots & 0 & 0 \\
    0 & 0 &\Phi_{r}\big(\tau,\bm D(\widetilde{H})\big)& 0\\
    0 & 0 &0 &\Phi_{r+1}\big(\tau,\bm D(\widetilde{H})\big)
    \end{matrix}\right]
  \end{equation}
  where $\bm D(Z)$ and $\bm T(Z)$  for any $Z=[Z_1\ \ Z_2]\in\mathbb{R}^{m\times n}$
  with $Z_1\in\mathbb{R}^{m\times m}$ are given by
  \begin{equation}\label{MD-Case2}
  \bm D(Z):=\big(\mathcal{G}(Z_{a_1a_1}),\mathcal{G}(Z_{a_2a_2}),\ldots,\mathcal{G}(Z_{a_{r}a_{r}}),[Z_{bb}\ \ Z_{bc}]\big)
  \end{equation}
  and
  \begin{equation}\label{MTZ-Case2}
   \bm T(Z):=\big[\mathcal{E}_1\circ\mathcal{G}(Z_1)+\mathcal{E}_2\circ\mathcal{H}(Z_1)\ \
             \mathcal{F}_{\alpha\cup\beta,c}\circ Z_{\alpha\cup\beta,c}\big].
  \end{equation}
  Here ``$\circ$'' means the Hardmard product.
  In particular, from \cite[Page 129]{DingC10}, we know that
  \begin{equation}\label{Proj-C2set}
   \Pi_{\mathcal{C}_2}(\tau,\bm D(\widetilde{H}))=
   \big(\Phi_{0}(\tau,\bm D(\widetilde{H})),\Phi_{1}(\tau,\bm D(\widetilde{H})),\ldots,
   \Phi_{r+1}(\tau,\bm D(\widetilde{H}))\big),
  \end{equation}
  where $\mathcal{C}_2\subseteq \mathbb{W}$ is a closed convex cone which, if $(t,X)\in{\rm bd}\,K$, has the form
 \begin{equation*}
   \mathcal{C}_2:=\Big\{(\zeta,W)\in\mathbb{W}\ |\ \sum_{l=1}^{r_0}{\rm tr}(W_l)+\|{\bm \kappa}_{\beta}(W)\|_{(k-k_0)}\le\zeta\Big\},
  \end{equation*}
  and otherwise takes the following form
 \begin{equation}\label{C2-cone}
  \mathcal{C}_2\!:=\!\Big\{(\zeta,W)\!\in\!\mathbb{W}\ | \sum_{l=1}^{r_0}{\rm tr}(W_l)\!+\|{\bm \kappa}_{\beta}(W)\|_{(k-k_0)}\le\zeta,
  \sum_{l=1}^{r_0}{\rm tr}(W_l)\!+\!\langle \overline{u}_\beta,{\bm\kappa}_{\beta}(W)\rangle\!=\!\zeta\Big\}.
  \end{equation}
 \end{alemma}

 The following lemma provides a crucial result on the directional derivative of $\Pi_K$,
 which can also be found from the proof of \cite[Proposition 4.5]{DingC10}.
%------------------------------------------------------------------------------------------------
 \begin{alemma}\label{Lemma2-pdird}
  Let $(t,X)\notin{\rm int}\,K\cup{\rm int}\,K^{\circ}$ be given.
  Write $(\overline{t},\overline{X})=\Pi_K(t,X)$ and $\overline{\sigma}=\sigma(\overline{X})$.
  Let $X$ have the SVD as in \eqref{X-SVD} with $a,b,c$ and $a_l\ (l=1,2,\ldots,r)$
  given by \eqref{abc} and \eqref{al}.
  \begin{enumerate}
  \item [(i)] If $\overline{\sigma}_k>0$, then $\Pi_{K}'\big((t,X);(\Delta t,\Delta X)+(\Delta\zeta,\Delta\Gamma)\big)=(\Delta\zeta,\Delta\Gamma)$ if and only if
         \begin{subnumcases}{}
          \label{equa0-pdird-case1}
         \Delta\zeta=\!\Phi_0\big(\tau,\bm D(\Delta\widetilde{X}+\Delta\widetilde{\Gamma})\big);\\
         \label{equa1-pdird-case1}
         \mathcal{G}(\Delta\widetilde{\Gamma}_{a_la_l})= \Phi_{l}\big(\tau,\bm D(\Delta\widetilde{X}+\Delta\widetilde{\Gamma})\big),\ \ l=1,2,\ldots,r_1;\\
        \label{equa2-pdird-case1}
         [\mathcal{G}(\Delta\widetilde{X}_1)]_{a_la_{l'}}=0,\ \ l\ne l'\ {\rm and}\ l,l'=1,2,\ldots,r_0;\\
         \label{equa3-pdird-case1}
         [\mathcal{G}(\Delta\widetilde{\Gamma}_1)]_{a_la_{l'}}=0,\ \ l\ne l'\ {\rm and}\ l,l'=r_0\!+\!1,\ldots,r_1;
        \end{subnumcases}
        and
       \begin{subnumcases}{}
       \label{equa4-pdird-case1}
        [\mathcal{G}(\Delta\widetilde{\Gamma}_1)]_{\alpha\beta}-\big(\mathcal{E}_1\big)_{\alpha\beta}\circ [\mathcal{G}(\Delta\widetilde{\Gamma}_1)]_{\alpha\beta}=\big(\mathcal{E}_1\big)_{\alpha\beta}\circ [\mathcal{G}(\Delta\widetilde{X}_1)]_{\alpha\beta};\\
        \label{equa5-pdird-case1}
        [\mathcal{G}(\Delta\widetilde{\Gamma}_1)]_{\alpha\gamma}-\big(\mathcal{E}_1\big)_{\alpha\gamma}\circ [\mathcal{G}(\Delta\widetilde{\Gamma}_1)]_{\alpha\gamma}=\big(\mathcal{E}_1\big)_{\alpha\gamma}\circ [\mathcal{G}(\Delta\widetilde{X}_1)]_{\alpha\gamma};\\
        \label{equa6-pdird-case1}
        [\mathcal{G}(\Delta\widetilde{\Gamma}_1)]_{\beta\gamma}-\big(\mathcal{E}_1\big)_{\beta\gamma}\circ [\mathcal{G}(\Delta\widetilde{\Gamma}_1)]_{\beta\gamma}=\big(\mathcal{E}_1\big)_{\beta\gamma}\circ [\mathcal{G}(\Delta\widetilde{X}_1)]_{\beta\gamma};\\
         \label{equa7-pdird-case1}
        \mathcal{H}(\Delta\widetilde{\Gamma}_1)\!-\widehat{\mathcal{E}}_2\circ \mathcal{H}(\Delta\widetilde{\Gamma}_1)
         \!=\widehat{\mathcal{E}}_2\circ \mathcal{H}(\Delta\widetilde{X}_1)\ {\rm with}\
        \widehat{\mathcal{E}}_2=\!\left[\begin{matrix}
        \big(\mathcal{E}_2\big)_{\overline{\gamma}\overline{\gamma}} & \big(\mathcal{E}_2\big)_{\overline{\gamma}\gamma}\\
        \big(\mathcal{E}_2\big)_{\gamma\overline{\gamma}} & 0
       \end{matrix}\right]\!; \\
        \label{cequa1-pdird}
         \Delta\widetilde{\Gamma}_{\overline{\gamma}c}-\mathcal{F}_{\overline{\gamma}c}\circ\Delta\widetilde{\Gamma}_{\overline{\gamma}c}
         =\mathcal{F}_{\overline{\gamma}c}\circ\Delta\widetilde{X}_{\overline{\gamma}c};\\
        \label{cequa2-pdird}
        [\Delta\widetilde{X}_{\gamma\gamma}\ \ \Delta\widetilde{X}_{\gamma c}]=0.
      \end{subnumcases}

 \item [(ii)] If $\overline{\sigma}_k=0$, then $\Pi_{K}'\big((t,X);(\Delta t,\Delta X)+(\Delta\zeta,\Delta\Gamma)\big)=(\Delta\zeta,\Delta\Gamma)$ if and only if
         \begin{subnumcases}{}
          \label{equa0-pdird-case2}
         \Delta\zeta=\!\Phi_0\big(\tau,\bm D(\Delta\widetilde{X}+\Delta\widetilde{\Gamma})\big);\\
         \label{equa1-pdird-case2}
         \big[\Delta\widetilde{\Gamma}_{bb}\ \ \Delta\widetilde{\Gamma}_{bc}\big]
         =\Phi_{r+1}\big(\tau,\bm D(\Delta\widetilde{X}+\Delta\widetilde{\Gamma})\big); \\
         \label{equa2-pdird-case2}
         \mathcal{G}(\Delta\widetilde{\Gamma}_{a_la_l})= \Phi_{l}\big(\tau,\bm D(\Delta\widetilde{X}+\Delta\widetilde{\Gamma})\big),
          \ \ l=1,2,\ldots,r; \\
         \label{equa3-pdird-case2}
         [\mathcal{G}(\Delta\widetilde{X}_1)]_{a_la_{l'}}=0,\ \ l\ne l'\ {\rm and}\ l,l'=1,2,\ldots,r_0;\\
         \label{equa4-pdird-case2}
          \Delta\widetilde{\Gamma}_{a_la_{l'}}=0,\ \ l\ne l'\ {\rm and}\ l,l'=r_0\!+\!1,\ldots,r\!+\!1;\\
         \label{equa5-pdird-case2}
         [\mathcal{G}(\Delta\widetilde{\Gamma}_1)]_{\alpha\beta}-\big(\mathcal{E}_1\big)_{\alpha\beta}\circ [\mathcal{G}(\Delta\widetilde{\Gamma}_1)]_{\alpha\beta}=\big(\mathcal{E}_1\big)_{\alpha\beta}\circ [\mathcal{G}(\Delta\widetilde{X}_1)]_{\alpha\beta};\\
         \label{equa6-pdird-case2}
          \Delta\widetilde{\Gamma}_{a_lc}=0,\ \ l=r_0\!+\!1,\ldots,r
        \end{subnumcases}
        and
        \begin{subnumcases}{}
        \label{equa7-pdird-case2}
        \mathcal{H}(\Delta\widetilde{\Gamma}_{\alpha\alpha})-\big(\mathcal{E}_2\big)_{\alpha\alpha}\circ \mathcal{H}(\Delta\widetilde{\Gamma}_{\alpha\alpha})=\big(\mathcal{E}_2\big)_{\alpha\alpha}\circ \mathcal{H}(\Delta\widetilde{X}_{\alpha\alpha});\\
        \label{equa8-pdird-case2}
        [\mathcal{H}(\Delta\widetilde{\Gamma}_1)]_{\alpha\beta}-\big(\mathcal{E}_2\big)_{\alpha\beta}\circ [\mathcal{H}(\Delta\widetilde{\Gamma}_1)]_{\alpha\beta}=\big(\mathcal{E}_2\big)_{\alpha\beta}\circ [\mathcal{H}(\Delta\widetilde{X}_1)]_{\alpha\beta};\\
         \label{equa9-pdird-case2}
        \Delta\widetilde{\Gamma}_{\alpha c}-\mathcal{F}_{\alpha c}\circ\Delta\widetilde{\Gamma}_{\alpha c}
         =\mathcal{F}_{\alpha c}\circ\Delta\widetilde{X}_{\alpha c}.
       \end{subnumcases}
  \end{enumerate}
 \end{alemma}
 \begin{proof}
  By Lemma \ref{Lemma1-pdird},
  $\Pi_{K}'\big((t,X);(\Delta t,\Delta X)+(\Delta\zeta,\Delta\Gamma)\big)=(\Delta\zeta,\Delta\Gamma)$
  if and only if
  \begin{equation}\label{equa-DetaXG}
    \Delta\zeta=\!\Phi_0\big(\tau,\bm D(\widetilde{H})\big)\ \ {\rm and}\ \
    \Delta\widetilde{\Gamma}=\Xi(\tau,\bm D(\widetilde{H}))\ \ {\rm with}\
    \tau=\Delta t\!+\!\Delta \zeta,\,H=\Delta X\!+\!\Delta\Gamma,
  \end{equation}
  where $\Xi(\tau,\bm D(\widetilde{H}))$ is given by equation \eqref{Xi-Case1} if $\overline{\sigma}_k>0$,
  and otherwise is given by \eqref{Xi-Case2}.

  \medskip
  \noindent
  {\bf(i)} It suffices to argue that the second equality of \eqref{equa-DetaXG} holds
  with $\Xi(\tau,\bm D(\widetilde{H}))$ given by \eqref{Xi-Case1} if and only if
  \eqref{equa1-pdird-case1}-\eqref{cequa2-pdird} hold.
  Suppose that the second equality of \eqref{equa-DetaXG} holds with
  $\Xi(\tau,\bm D(\widetilde{H}))$ given by \eqref{Xi-Case1}.
  We immediately have \eqref{cequa1-pdird}. Notice that $(\mathcal{E}_1)_{\gamma\gamma}=0$.
  By the expressions of $(\mathcal{E}_2)_{\gamma\gamma}$ and $\mathcal{F}_{\gamma c}$,
  we obtain \eqref{cequa2-pdird}. By the symmetry of $\mathcal{E}_1$ and
  $\Phi_l(\tau,\bm D(\widetilde{H}))$,
  \begin{equation}\label{mainequa1-case1}
  \mathcal{H}(\Delta\widetilde{\Gamma}_1)=\mathcal{E}_2\circ\mathcal{H}(\Delta\widetilde{X}_1+\Delta\widetilde{\Gamma}_1).
  \end{equation}
  This implies \eqref{equa7-pdird-case1}. By the symmetry of $\mathcal{E}_2$
  and $\Phi_l(\tau,\bm D(\widetilde{H}))$, we have that
  \begin{equation}
  \label{mainequa2-case1}
   \mathcal{G}(\Delta\widetilde{\Gamma}_1)=\mathcal{E}_1\circ\mathcal{G}(\Delta\widetilde{H}_1)
   +\!\left[\begin{matrix}
    \Phi_1\big(\tau,\bm D(\widetilde{H})\big) & 0 & 0 &0\\
    0 & \ddots & 0 & 0 \\
    0 & 0 &\Phi_{r_1}\big(\tau,\bm D(\widetilde{H})\big)& 0 \\
    0 & 0 &0 & \mathcal{G}(\widetilde{H}_{\gamma\gamma})
    \end{matrix}\right].
  \end{equation}
  By this, we obtain equations \eqref{equa1-pdird-case1}-\eqref{equa6-pdird-case1}.
  By using $\overline{\sigma}=\sigma-\theta\overline{u}$ and \eqref{sigmabar-case1},
  we have that
  \begin{subnumcases}{}\label{ME1-equa1-case1}
    [\mathcal{E}_1]_{a_la_{l'}}=E_{a_la_{l'}},\ {\rm for}\ l\ne l'\
    {\rm and}\ l,l'\!\in\!\{1,\ldots,r_0\}\ {\rm or}\ l,l'\!\in\!\{r_1\!+\!1,\ldots,r\!+\!1\}; \nonumber\\
        \label{ME1-equa2-case1}
    [\mathcal{E}_1]_{a_la_{l'}}=0,\ \ {\rm for}\ l\ne l'\
    {\rm and}\ l,l'\!\in\!\{r_0+1,\ldots,r_1\}; \nonumber\\
       \label{ME1-equa3-case1}
    [\mathcal{E}_1]_{a_la_l}=0,\ \ {\rm for}\ l=1,2,\ldots,r\!+\!1.\nonumber
  \end{subnumcases}
  Together with \eqref{mainequa2-case1}, we obtain that
  \eqref{equa1-pdird-case1}-\eqref{equa3-pdird-case1} hold.
  Consequently, the necessity follows. Conversely, suppose that
  \eqref{equa1-pdird-case1}-\eqref{equa3-pdird-case1} and
  \eqref{equa4-pdird-case1}-\eqref{equa7-pdird-case1} hold.
  Then \eqref{mainequa1-case1} and \eqref{mainequa2-case1} must hold.
  Along with \eqref{cequa1-pdird}-\eqref{cequa2-pdird},
  we obtain the second equality of \eqref{equa-DetaXG}.

  \medskip
  \noindent
   {\bf(ii)} It suffices to argue that the second equality of \eqref{equa-DetaXG} holds
  with $\Xi(\tau,\bm D(\widetilde{H}))$ given by \eqref{Xi-Case2} if and only if
  equations \eqref{equa1-pdird-case2}-\eqref{equa9-pdird-case2} hold.
  Suppose that the second equality of \eqref{equa-DetaXG} holds with
  $\Xi(\tau,\bm D(\widetilde{H}))$ given by \eqref{Xi-Case2}.
  Firstly, we readily get equation \eqref{equa9-pdird-case2}.
  By noting that $(\mathcal{E}_1)_{bb}=0$, $(\mathcal{E}_2)_{bb}=0$ and
  $\mathcal{F}_{\beta c}=0$, equation \eqref{equa1-pdird-case2} and
  \eqref{equa6-pdird-case2} also hold. Notice that
  \begin{equation}\label{mainequa1-case2}
  \mathcal{H}(\Delta\widetilde{\Gamma}_1)
  =\mathcal{E}_2\circ\mathcal{H}(\Delta\widetilde{H}_1)
   +\left[\begin{matrix}
    0_{\alpha\cup(\beta\backslash b),\alpha\cup(\beta\backslash b)} & 0_{\alpha\cup(\beta\backslash b),b} \\
    0_{b,\alpha\cup(\beta\backslash b)} & \mathcal{H}([\Phi_{r+1}\big(\tau,\bm D(\widetilde{H})\big)]_1) \\
    \end{matrix}\right]
  \end{equation}
  and
  \begin{equation}
  \label{mainequa2-case2}
   \mathcal{G}(\Delta\widetilde{\Gamma}_1)\!=\mathcal{E}_1\circ\mathcal{G}(\widetilde{H}_1)
   +\!\left[\begin{matrix}
    \Phi_1\big(\tau,\bm D(\widetilde{H})\big) & 0 & 0 &0 \\
    0 & \ddots & 0  &0\\
    0 & 0 &\Phi_{r}\big(\tau,\bm D(\widetilde{H})\big)& 0 \\
    0 & 0 &0 &\mathcal{G}([\Phi_{r+1}\big(\tau,\bm D(\widetilde{H})\big)]_1) \\
    \end{matrix}\right]\\
  \end{equation}
  due to the symmetry of $\mathcal{E}_2$, $\mathcal{E}_1$ and
  $\Phi_l\big(\tau,\bm D(\widetilde{H})\big)$ for $l=1,\ldots,r$,
  where $[\Phi_{r+1}\big(\tau,\bm D(\widetilde{H})\big)]_1$ is the matrix
  consisting of the first $m$ columns of $\Phi_{r+1}\big(\tau,\bm D(\widetilde{H})\big)$.
  Equation \eqref{mainequa1-case2} implies that \eqref{equa7-pdird-case2}-\eqref{equa8-pdird-case2}
  hold, while \eqref{mainequa2-case2} implies that \eqref{equa3-pdird-case2} holds.
  By $\overline{\sigma}=\sigma-\theta\overline{u}$ and \eqref{sigmabar-case2},
  \begin{subnumcases}{}\label{ME1-equa1-case2}
    [\mathcal{E}_1]_{a_la_{l'}}=E_{a_la_{l'}},\ {\rm for}\ l\ne l'\
    {\rm and}\ l,l'\!\in\!\{1,2,\ldots,r_0\}; \nonumber\\
        \label{ME1-equa2-case2}
    [\mathcal{E}_1]_{a_la_{l'}}=0,\ \ {\rm for}\ l\ne l'\
    {\rm and}\ l,l'\!\in\!\{r_0\!+\!1,\ldots,r\!+\!1\}; \nonumber\\
       \label{ME1-equa3-case2}
    [\mathcal{E}_1]_{a_la_l}=0,\ \ {\rm for}\ l=1,2,\ldots,r\!+\!1.\nonumber
  \end{subnumcases}
  Together with \eqref{mainequa2-case2}, we get equations \eqref{equa2-pdird-case2}-\eqref{equa5-pdird-case2}.
  Consequently, the necessity follows. Conversely, suppose that \eqref{equa2-pdird-case2}-\eqref{equa5-pdird-case2}
  and \eqref{equa7-pdird-case2}-\eqref{equa8-pdird-case2} hold.
  Then \eqref{mainequa1-case2} and \eqref{mainequa2-case2} must hold.
  Together with \eqref{equa1-pdird-case2}, \eqref{equa6-pdird-case2} and
  \eqref{equa9-pdird-case2}, we obtain the second equality of \eqref{equa-DetaXG}.
 \end{proof}

  Next we provide an equivalent characterization for the critical cone of $K$ at $(\overline{t},\overline{X})+(\overline{\zeta},\overline{\Gamma})$
  associated with the complementarity problem $K\ni(t,X)\perp(\tau,Y)\in K^{\circ}$.
  %------------------------------------------------------------------------------------
 \begin{alemma}\label{Lemma1-critical-cone}
  Let $((\overline{t},\overline{X}),(\overline{\zeta},\overline{\Gamma}))\in{\rm gph}\mathcal{N}_{K}$ be given
  with $(\overline{t}+\overline{\zeta},\overline{X}+\overline{\Gamma})\notin{\rm int}\,K\cup{\rm int}\,K^{\circ}$.
  We write $(t,X)=(\overline{t},\overline{X})+(\overline{\zeta},\overline{\Gamma})$,
  $\overline{\sigma}=\sigma(\overline{X})$ and $\sigma=\sigma(X)$.
  Let $X$ have the SVD as in \eqref{X-SVD} with the index sets $a,b,c$
  and $a_l\ (l=1,2,\ldots,r)$ defined by \eqref{abc} and \eqref{al}.
  \begin{enumerate}
    \item [(i)] If $\overline{\sigma}_k>0$, then $(\tau,Z)\in\mathcal{T}_K(\overline{t},\overline{X})\cap(\overline{\zeta},\overline{\Gamma})^{\perp}$
                if and only if $(\tau,\bm D(\widetilde{Z}))\in \mathcal{C}_1$.
    \item[(ii)] If $\overline{\sigma}_k=0$, then $(\tau,Z)\in\mathcal{T}_K(\overline{t},\overline{X})
                \cap(\overline{\zeta},\overline{\Gamma})^{\perp}$
                if and only if $(\tau,\bm D(\widetilde{Z}))\in \mathcal{C}_2$.
 \end{enumerate}
 \end{alemma}
 \begin{proof}
  {\bf(i)} Let $r_0,r_1\!\in\{1,2,\ldots,r\}$
  be such that $\alpha=\bigcup_{l=1}^{r_0}a_l$, $\beta=\bigcup_{l=r_0+1}^{r_1}a_l$
  and $\gamma=\bigcup_{l=r_1+1}^{r+1}a_l$, where $\alpha,\beta,\gamma$ and
  $\overline{\gamma}$ are defined by \eqref{abg-case1} with
  $k_0\in[0,k\!-\!1]$ and $k_1\in[k,m]$ such that \eqref{sigmabar-case1} holds.
  Also, by $\overline{\sigma}=\sigma-\theta\overline{u}$ (see Case (i) of Lemma \ref{Lemma1-pdird})
  and $(\overline{\zeta},\overline{\Gamma})\in{\rm bd}\,K^{\circ}$,
  \begin{equation}\label{WGamma-SVD1}
   \overline{\zeta}=t-\overline{t}=-\theta\ \ {\rm and}\ \ \overline{\Gamma}=X-\overline{X}=\overline{U}\big[{\rm Diag}(\theta\overline{u})\ \ 0\big]\overline{V}^{\mathbb{T}}.
  \end{equation}
  Let $(\tau,Z)$ be an arbitrary point from
  $\mathcal{T}_K(\overline{t},\overline{X})\cap (\overline{\zeta},\overline{\Gamma})^{\perp}$.
  From $(\tau,Z)\in\mathcal{T}_K(\overline{t},\overline{X})$ and the expression of
  $\mathcal{T}_K(\overline{t},\overline{X})$ (see \cite[Equation(4.1)]{DingC10}), it follows that
  \[
    {\textstyle\sum_{l=1}^{r_0}}{\rm tr}(\overline{U}_{a_l}^{\mathbb{T}}Z\overline{V}_{a_l})
    +{\textstyle\sum_{i=1}^{k-k_0}}\lambda_i\big(\mathcal{G}(\widetilde{Z}_{\beta\beta})\big)\le \tau,
  \]
 while from $(\tau,Z)\in(\overline{\zeta},\overline{\Gamma})^{\perp}$ and equation \eqref{WGamma-SVD1}, it follows that
 \begin{align*}
   0=\tau\overline{\zeta}+\langle Z,\overline{\Gamma}\rangle
    &=-\theta\tau+\theta{\textstyle\sum_{l=1}^{r_0}}{\rm tr}(\overline{U}_{a_l}^{\mathbb{T}}Z\overline{V}_{a_l})
     +\theta\langle \widetilde{Z}_{\beta\beta}, {\rm Diag}(\overline{u}_{\beta})\rangle\\
    &=-\theta\tau+\theta{\rm tr}(\overline{U}_{\alpha}^{\mathbb{T}}Z\overline{V}_{\alpha})
      +\theta\langle\mathcal{G}(\widetilde{Z}_{\beta\beta}),{\rm Diag}(\overline{u}_{\beta})\rangle\\
    &\le -\theta\tau+\theta{\rm tr}(\overline{U}_{\alpha}^{\mathbb{T}}Z\overline{V}_{\alpha})
        +\theta\overline{u}_{\beta}^{\mathbb{T}}\lambda(\mathcal{G}(\widetilde{Z}_{\beta\beta}))\\
    &\le -\theta\tau+\theta{\rm tr}(\overline{U}_{\alpha}^{\mathbb{T}}Z\overline{V}_{\alpha})
        +\theta{\textstyle\sum_{i=1}^{k-k_0}}\lambda_i\big(\mathcal{G}(\widetilde{Z}_{\beta\beta})\big)
 \end{align*}
 where the first inequality is using the von Neumann's trace inequality,
 and the last one is due to $\overline{u}_{\beta}=\overline{u}_{\beta}^{\downarrow}$
 and $\sum_{i\in\beta}\overline{u}_i=k-k_0$. So,
 $(\tau,Z)\in\mathcal{T}_K(\overline{t},\overline{X})\cap(\overline{\zeta},\overline{\Gamma})^{\perp}$
 if and only if
 \begin{equation}\label{eigenvalue-ineq}
  \langle\mathcal{G}(\widetilde{Z}_{\beta\beta}),{\rm Diag}(\overline{u}_{\beta})\rangle
  =\overline{u}_{\beta}^{\mathbb{T}}\lambda(\mathcal{G}(\widetilde{Z}_{\beta\beta}))
 \end{equation}
 and
 \begin{equation}\label{C1set-ineq}
    {\rm tr}(\widetilde{Z}_{\alpha\alpha})
   +{\textstyle\sum_{i=1}^{k-k_0}}\lambda_i\big(\mathcal{G}(\widetilde{Z}_{\beta\beta}\big)\le \tau
   \ {\rm and}\
   {\rm tr}(\widetilde{Z}_{\alpha\alpha})
   +\langle \overline{u}_{\beta}, \lambda(\mathcal{G}(\widetilde{Z}_{\beta\beta}))\rangle
   =\tau.
 \end{equation}
 By the von Neumann's trace inequality, equation \eqref{eigenvalue-ineq}
 is equivalent to saying that $\mathcal{G}(\widetilde{Z}_{\beta\beta})$ and
 ${\rm Diag}(\overline{u}_{\beta})$ have a simultaneous ordered eigenvalue decomposition,
 which by the proof of Case (i) of \cite[Proposition 5.1]{DingC15} is equivalent to
 saying that
 \[
   \mathcal{G}(\widetilde{Z}_{\beta\beta})
   =\left[\begin{matrix}
     \mathcal{G}(\widetilde{Z}_{\beta_1\beta_1}) & 0 & 0\\
      0 & \tau I_{|\beta_2|} & 0\\
      0 & 0 &\mathcal{G}(\widetilde{H}_{\beta_3\beta_3})
      \end{matrix}\right],
  \]
  where $\beta_1$, $\beta_2$ and $\beta_3$ are the index sets defined as in Lemma \ref{k0k1-propery}.
 Together with \eqref{C1set-ineq}, the definitions of the operator $\bm D$ in \eqref{MD-Case1}
 and the set $\mathcal{C}_1$ in \eqref{C1-cone}, we get the desired equivalence.

  \medskip
  \noindent
  {\bf(ii)} Let $r_0\in\{1,\ldots,r\}$ be such that $\alpha=\bigcup_{l=1}^{r_0}a_l$
  and $\beta=\bigcup_{l=r_0+1}^{r+1}a_l$, where $\alpha$ and $\beta$ are defined by \eqref{abg-case2}
  with $k_0\in[0,k\!-\!1]$ such that \eqref{sigmabar-case2} holds. Now
  $\beta_1=a_{r_0+1}$, $\beta_3=b$ and $\beta_2=\bigcup_{l=r_0+2}^{r}a_l$. Also,
  by $\overline{\sigma}=\sigma-\theta\overline{u}$ (see Case (ii) of Lemma \ref{Lemma1-pdird})
  and $(\overline{\zeta},\overline{\Gamma})\in{\rm bd}\,K^{\circ}$,
  \begin{equation}\label{WGamma-SVD2}
   \overline{\zeta}=t-\overline{t}=-\theta\ \ {\rm and}\ \ \overline{\Gamma}=X-\overline{X}=\overline{U}\big[{\rm Diag}(\theta\overline{u})\ \ 0\big]\overline{V}^{\mathbb{T}}.
  \end{equation}
  Let $(\tau,Z)$ be an arbitrary point from
  $\mathcal{T}_K(\overline{t},\overline{X})\cap (\overline{\zeta},\overline{\Gamma})^{\perp}$.
  From $(\tau,Z)\in\mathcal{T}_K(\overline{t},\overline{X})$ and the expression of
  $\mathcal{T}_K(\overline{t},\overline{X})$ (see \cite[Equation(4.2)]{DingC10}), it follows that
  \[
    {\textstyle\sum_{l=1}^{r_0}}{\rm tr}(\overline{U}_{a_l}^{\mathbb{T}}Z\overline{V}_{a_l})
   +{\textstyle\sum_{i=1}^{k-k_0}}\sigma_i\big([\widetilde{Z}_{\beta\beta}\ \ \widetilde{Z}_{\beta c}]\big)\le \tau,
  \]
 while from $(\tau,Z)\in(\overline{\zeta},\overline{\Gamma})^{\perp}$ and equation \eqref{WGamma-SVD2},
 it follows that
 \begin{align*}
   0=\tau\zeta+\langle Z,\overline{\Gamma}\rangle
    &=-\theta\tau+\theta{\textstyle\sum_{l=1}^{r_0}}{\rm tr}(\overline{U}_{a_l}^{\mathbb{T}}Z\overline{V}_{a_l})
     +\theta\langle [\widetilde{Z}_{\beta\beta}\ \ \widetilde{Z}_{\beta c}],[{\rm Diag}(\overline{u}_{\beta})\ \ 0_{\beta c}]\rangle\\
    &\le -\theta\tau+\theta{\rm tr}(\overline{U}_{\alpha}^{\mathbb{T}}Z\overline{V}_{\alpha})
        +\theta\overline{u}_{\beta}^{\mathbb{T}}\sigma([\widetilde{Z}_{\beta\beta}\ \ \widetilde{Z}_{\beta c}])\\
    &\le -\theta\tau+\theta{\rm tr}(\overline{U}_{\alpha}^{\mathbb{T}}Z\overline{V}_{\alpha})
        +\theta{\textstyle\sum_{i=1}^{k-k_0}}\sigma_i([\widetilde{Z}_{\beta\beta}\ \ \widetilde{Z}_{\beta c}])
 \end{align*}
 where the first inequality is using the von Neumann's trace inequality,
 and the last one is due to $\overline{u}_{\beta}=\overline{u}_{\beta}^{\downarrow}$
 and $\sum_{i\in\beta}\overline{u}_i=k-k_0$. So,
 $(\tau,Z)\in\mathcal{T}_K(\overline{t},\overline{X})\cap(\overline{\zeta},\overline{\Gamma})^{\perp}$
 if and only if
 \begin{equation}\label{singularvalue-ineq}
  \langle [\widetilde{Z}_{\beta\beta}\ \ \widetilde{Z}_{\beta c}],[{\rm Diag}(\overline{u}_{\beta})\ \ 0_{\beta c}]\rangle
  = \overline{u}_{\beta}^{\mathbb{T}}\sigma([\widetilde{Z}_{\beta\beta}\ \ \widetilde{Z}_{\beta c}])
 \end{equation}
 and
 \begin{equation}\label{C2set-ineq}
   {\rm tr}(\widetilde{Z}_{\alpha\alpha})
   +{\textstyle\sum_{i=1}^{k-k_0}}\sigma_i\big([\widetilde{Z}_{\beta\beta}\ \ \widetilde{Z}_{\beta c}]\big)
   \le \tau
   \ {\rm and}\
   {\rm tr}(\widetilde{Z}_{\alpha\alpha})
   +\langle \overline{u}_{\beta},\sigma([\widetilde{Z}_{\beta\beta}\ \ \widetilde{Z}_{\beta c}])\rangle
   =\tau.
 \end{equation}
 By the von Neumann's trace inequality, equation \eqref{singularvalue-ineq}
 is equivalent to saying that $[\widetilde{Z}_{\beta\beta}\ \ \widetilde{Z}_{\beta c}]$ and
 $[{\rm Diag}(\overline{u}_{\beta})\ \ 0_{\beta c}]$ have a simultaneous ordered SVD,
 which by the proof of Case (ii) of \cite[Proposition 5.1]{DingC15} is equivalent to
 saying that $[\widetilde{Z}_{\beta\beta}\ \ \widetilde{Z}_{\beta c}]$ has the following structure
 \begin{equation}\label{Zbb-case2}
   [\widetilde{Z}_{\beta\beta}\ \ \widetilde{Z}_{\beta c}]
   =\left[\begin{matrix}
     \widetilde{Z}_{a_{r_0+1}a_{r_0+1}} & 0 & \cdots & 0 & 0 & 0\\
      0 & \widetilde{Z}_{a_{r_0+2}a_{r_0+2}} &\cdots & 0 & 0 & 0\\
      \vdots & \vdots &\ddots&\vdots&\vdots&\vdots\\
      0 &0 & \cdots&\widetilde{Z}_{a_ra_r}& 0& 0\\
      0 &0 & \cdots& 0 & \widetilde{Z}_{bb}& \widetilde{Z}_{bc}\\
      \end{matrix}\right]
  \end{equation}
 with $ \widetilde{Z}_{a_la_l}\in\mathbb{S}^{|a_l|}$ for $l=r_0+1,\ldots,r$.
 Together with equation \eqref{C2set-ineq}, the definitions of the operator $\bm D$
 in \eqref{MD-Case2} and the set $\mathcal{C}_2$ in \eqref{C2-cone},
 we obtain the desired equivalence.
 \end{proof}

  Now we give a characterization for the negative polar cone of the critical cone of $K$
  at $(\overline{t},\overline{X})+(\overline{\zeta},\overline{\Gamma})$
  associated with the complementarity problem $K\ni(t,X)\perp(\tau,Y)\in K^{\circ}$.
%---------------------------------------------------------------------------------------------
 \begin{alemma}\label{Lemma2-critical-cone}
  Let $((\overline{t},\overline{X}),(\overline{\zeta},\overline{\Gamma}))\in{\rm gph}\mathcal{N}_{K}$ be given
  with $(\overline{t}+\overline{\zeta},\overline{X}+\overline{\Gamma})\notin{\rm int}\,K\cup{\rm int}\,K^{\circ}$.
  We write $(t,X)=(\overline{t},\overline{X})+(\overline{\zeta},\overline{\Gamma})$,
  $\overline{\sigma}=\sigma(\overline{X})$ and $\sigma=\sigma(X)$.
  Let $X$ have the SVD as in \eqref{X-SVD} with the index sets $a,b,c$
  and $a_l\ (l=1,2,\ldots,r)$ defined by \eqref{abc} and \eqref{al}.
  Suppose that $((\Delta t,\Delta X),(\Delta\zeta,\Delta\Gamma))\in{\rm gph}D\mathcal{N}_{K^\circ}
  \big((\overline{\zeta},\overline{\Gamma})|(\overline{t},\overline{X})\big)$.
  The following statements hold.
  \begin{enumerate}
    \item [(i)] If $\overline{\sigma}_k>0$, then $(\Delta t,\Delta X)-
                (0,\overline{U}\mathfrak{X}(\Delta\widetilde{\Gamma})\overline{V}^{\mathbb{T}})
                \in[\mathcal{T}_K(\overline{t},\overline{X})\cap (\overline{\zeta},\overline{\Gamma})^{\perp}]^{\circ}$
                if and only if
                \begin{equation}\label{npcone-case1}
                 \tau \Delta t+{\textstyle\sum_{l=1}^{r_1}}\langle \mathcal{G}(\Delta\widetilde{X}_{a_la_l}),
                 \mathcal{G}(\widetilde{Z}_{a_la_l})\rangle\le 0
                \quad\ \forall(\tau,Z)\in\mathcal{T}_K(\overline{t},\overline{X})\cap (\overline{\zeta},\overline{\Gamma})^{\perp},
                \end{equation}
                where $\mathfrak{X}\!: \mathbb{R}^{m\times n}\to\mathbb{R}^{m\times n}$ is a mapping
                defined by \eqref{MFX-Case1}.

  \item [(ii)] If $\overline{\sigma}_k=0$, then $(\Delta t,\Delta X)-
               (0,\overline{U}\mathfrak{X}(\Delta\widetilde{\Gamma})\overline{V}^{\mathbb{T}})
                \in[\mathcal{T}_K(\overline{t},\overline{X})\cap (\overline{\zeta},\overline{\Gamma})^{\perp}]^{\circ}$
                if and only if
                \begin{equation}\label{npcone-case2}
                 \tau \Delta t+{\textstyle\sum_{l=1}^{r}}\langle \mathcal{G}(\Delta\widetilde{X}_{a_la_l}),\mathcal{G}(\widetilde{Z}_{a_la_l})\rangle
                  +\langle\Delta\widetilde{X}_{b b},\widetilde{Z}_{bb}\rangle
                  +\langle \Delta\widetilde{X}_{bc},\widetilde{Z}_{bc}\rangle\le 0
                \end{equation}
                for any $(\tau,Z)\in\mathcal{T}_K(\overline{t},\overline{X})\cap (\overline{\zeta},\overline{\Gamma})^{\perp}$,
                where $\mathfrak{X}\!: \mathbb{R}^{m\times n}\to\mathbb{R}^{m\times n}$
                is defined by \eqref{MFX-Case2}.
 \end{enumerate}
 \end{alemma}
 \begin{proof}
 Since $((\Delta t,\Delta X),(\Delta\zeta,\Delta\Gamma))\in{\rm gph}D\mathcal{N}_{K^\circ}
 \big((\overline{\zeta},\overline{\Gamma})|(\overline{t},\overline{X})\big)$,
 by Lemma \ref{tcone-lemma} we have that
 $\Pi_{K}'\big((t,X);(\Delta t,\Delta X)+(\Delta\zeta,\Delta\Gamma)\big)=(\Delta\zeta,\Delta\Gamma)$.
 So, the results of Lemma \ref{Lemma2-pdird} hold.

 \medskip
 \noindent
 {\bf(i)} Let $r_0,r_1\!\in\{1,2,\ldots,r\}$ be such that $\alpha=\bigcup_{l=1}^{r_0}a_l$,
 $\beta=\bigcup_{l=r_0+1}^{r_1}a_l$ and $\gamma=\bigcup_{l=r_1+1}^{r+1}a_l$,
 where $\alpha,\beta,\gamma$ and $\overline{\gamma}$ are defined by \eqref{abg-case1}
 with the integers $k_0\in[0,k\!-\!1]$ and $k_1\in[k,m]$ such that \eqref{sigmabar-case1} holds.
 For any $(\tau,Z)\in\mathbb{X}$, from equation \eqref{cequa2-pdird}, we have that
  \begin{align*}
    &\tau\Delta t+\langle \Delta X,Z\rangle
    =\tau\Delta t+\langle \Delta\widetilde{X},\widetilde{Z}\rangle
    =\tau\Delta t +\langle \Delta\widetilde{X}_1,\widetilde{Z}_1\rangle
      +\langle \Delta\widetilde{X}_2,\widetilde{Z}_2\rangle\nonumber\\
    &=\tau \Delta t +\!\langle \mathcal{G}(\Delta\widetilde{X}_1),\mathcal{G}(\widetilde{Z}_1)\rangle
      +\!\langle \mathcal{H}(\Delta\widetilde{X}_1),\mathcal{H}(\widetilde{Z}_1)\rangle
      +\langle \Delta\widetilde{X}_{\alpha c},\widetilde{Z}_{\alpha c}\rangle
      +\langle \Delta\widetilde{X}_{\beta c},\widetilde{Z}_{\beta c}\rangle \nonumber\\
   &=\tau \Delta t+\sum_{l=1}^{r_1}\langle \mathcal{G}(\Delta\widetilde{X}_{a_la_l}),\mathcal{G}(\widetilde{Z}_{a_la_l})\rangle
        +\langle \mathcal{H}(\Delta\widetilde{X}_1),\mathcal{H}(\widetilde{Z}_1)\rangle
        +\langle \Delta\widetilde{X}_{\alpha c},\widetilde{Z}_{\alpha c}\rangle
      +\langle \Delta\widetilde{X}_{\beta c},\widetilde{Z}_{\beta c}\rangle\nonumber\\
    &\quad +2\big[\langle [\mathcal{G}(\Delta\widetilde{X}_1)]_{\alpha\beta},[\mathcal{G}(\widetilde{Z}_1)]_{\alpha\beta}\rangle
          \!+\!\langle[\mathcal{G}(\Delta\widetilde{X}_1)]_{\alpha\gamma},[\mathcal{G}(\widetilde{Z}_1)]_{\alpha\gamma}\rangle
          \!+\!\langle [\mathcal{G}(\Delta\widetilde{X}_1)]_{\beta\gamma},
          [\mathcal{G}(\widetilde{Z}_1)]_{\beta\gamma}\rangle\big]
  \end{align*}
  where the third equality is using $\langle \mathcal{G}(\Delta\widetilde{X}_1),\mathcal{H}(\widetilde{Z}_1)\rangle=0$.
  By \eqref{equa7-pdird-case1} it follows that
  \[
  \langle [\mathcal{H}(\Delta\widetilde{X}_1)]_{a_la_{l'}},[\mathcal{H}(\widetilde{Z}_1)]_{a_la_{l'}}\rangle
  =\langle (\Theta_2)_{a_la_{l'}}\circ[\mathcal{H}(\Delta\widetilde{\Gamma}_1)]_{a_la_{l'}},
   [\mathcal{H}(\widetilde{Z}_1)]_{a_la_{l'}}\rangle
  \]
  for $l=1,\ldots,r_1$ and $l'=1,2,\ldots,r\!+\!1$,
  or $l=r_1\!+\!1,\ldots,r\!+\!1$ and $l'=1,\ldots,r_1$, where
  \(
   (\Theta_2)_{a_la_{l'}}=[E_{a_la_{l'}}-(\mathcal{E}_2)_{a_la_{l'}}]
    \oslash(\mathcal{E}_2)_{a_la_{l'}}
  \)
  with ``$\oslash$'' denoting the division of entries.
  While from equations \eqref{equa4-pdird-case1}-\eqref{equa6-pdird-case1} and
  equation \eqref{cequa1-pdird} it immediately follows that
  \begin{numcases}{}
  \langle [\mathcal{G}(\Delta\widetilde{X}_1)]_{a_la_{l'}},[\mathcal{G}(\widetilde{Z}_1)]_{a_la_{l'}}\rangle
  =\langle (\Theta_1)_{a_la_{l'}}\circ[\mathcal{G}(\Delta\widetilde{\Gamma}_1)]_{a_la_{l'}},
   [\mathcal{G}(\widetilde{Z}_1)]_{a_la_{l'}}\rangle,\nonumber\\
   \langle \Delta\widetilde{X}_{\alpha c},\widetilde{Z}_{\gamma c}\rangle
   =\big\langle[(E-\mathcal{F}_{\alpha c})\oslash\mathcal{F}_{\alpha c}]\circ\Delta\widetilde{\Gamma}_{\alpha c},
     \widetilde{Z}_{\alpha c}\big\rangle,\nonumber\\
    \langle \Delta\widetilde{X}_{\beta c},\widetilde{Z}_{\beta c}\rangle
   =\big\langle[(E-\mathcal{F}_{\beta c})\oslash\mathcal{F}_{\beta c}]\circ\Delta\widetilde{\Gamma}_{\beta c},
     \widetilde{Z}_{\beta c}\big\rangle.\nonumber
  \end{numcases}{}
 for $l,l'=1,2,\ldots,r\!+\!1$ and $l\ne l'$, where
  \(
   (\Theta_1)_{a_la_{l'}}=[E_{a_la_{l'}}-(\mathcal{E}_1)_{a_la_{l'}}]
    \oslash(\mathcal{E}_1)_{a_la_{l'}}.
  \)
  From the last two groups of equalities, an elementary calculation yields that
  \begin{align*}
    &\langle \mathcal{H}(\Delta\widetilde{X}_1),\mathcal{H}(\widetilde{Z}_1)\rangle
        +\langle \Delta\widetilde{X}_{\alpha c},\widetilde{Z}_{\alpha c}\rangle+\langle \Delta\widetilde{X}_{\beta c},\widetilde{Z}_{\beta c}\rangle
    +2\langle [\mathcal{G}(\Delta\widetilde{X}_1)]_{\alpha\beta},[\mathcal{G}(\widetilde{Z}_1)]_{\alpha\beta}\rangle\\
   &+2\big[\langle[\mathcal{G}(\Delta\widetilde{X}_1)]_{\alpha\gamma},[\mathcal{G}(\widetilde{Z}_1)]_{\alpha\gamma}\rangle
          \!+\!\langle [\mathcal{G}(\Delta\widetilde{X}_1)]_{\beta\gamma},
          [\mathcal{G}(\widetilde{Z}_1)]_{\beta\gamma}\rangle\big]\\
   &=\langle\mathfrak{X}(\Delta\widetilde{\Gamma}),\widetilde{Z}\rangle
   =\langle \overline{U}\mathfrak{X}(\Delta\widetilde{\Gamma})\overline{V}^{\mathbb{T}},Z\rangle
  \end{align*}
  where $\mathfrak{X}\!: \mathbb{R}^{m\times n}\to\mathbb{R}^{m\times n}$ is a mapping
  with $\mathfrak{X}(A)$ for any $A\in\mathbb{R}^{m\times n}$ defined by
  \begin{align}\label{MFX-Case1}
   \mathfrak{X}(A)&:=\left[\begin{matrix}
     0_{\alpha\alpha}& (\Theta_1)_{\alpha\beta}\circ [\mathcal{G}(A_1)]_{\alpha\beta}
     &(\Theta_1)_{\alpha\gamma}\circ [\mathcal{G}(A_1)]_{\alpha\gamma}
     &\widehat{\mathcal{F}}_{\alpha c}\circ A_{\alpha c}\\
     (\Theta_1)_{\beta\alpha}\circ [\mathcal{G}(A_1)]_{\beta\alpha}&0_{\beta\beta}
     &(\Theta_1)_{\beta\gamma}\circ [\mathcal{G}(A_1)]_{\beta\gamma}
     &\widehat{\mathcal{F}}_{\beta c}\circ A_{\beta c}\\
     (\Theta_1)_{\gamma\alpha}\circ [\mathcal{G}(A_1)]_{\gamma\alpha}
     &(\Theta_1)_{\gamma\beta}\circ [\mathcal{G}(A_1)]_{\gamma\beta}&0_{\gamma\gamma}
     &0_{\gamma c}\nonumber\\
    \end{matrix}\right]\\
   &\quad+\left[\begin{matrix}
     (\Theta_2)_{\alpha\alpha}\circ \mathcal{H}(A_{\alpha\alpha})& (\Theta_2)_{\alpha\beta}\circ [\mathcal{H}(A_1)]_{\alpha\beta}
     &(\Theta_2)_{\alpha\gamma}\circ [\mathcal{H}(A_1)]_{\alpha\gamma}
     &0_{\alpha c}\\
     (\Theta_2)_{\beta\alpha}\circ [\mathcal{H}(A_1)]_{\beta\alpha}& (\Theta_2)_{\beta\beta}\circ \mathcal{H}(A_{\beta\beta})
     &(\Theta_2)_{\beta\gamma}\circ [\mathcal{H}(A_1)]_{\beta\gamma}
     &0_{\beta c}\\
     (\Theta_2)_{\gamma\alpha}\circ [\mathcal{G}(A_1)]_{\gamma\alpha}
     &(\Theta_2)_{\gamma\beta}\circ [\mathcal{G}(A_1)]_{\gamma\beta}&0_{\gamma\gamma}&0_{\gamma c}\\
    \end{matrix}\right]
  \end{align}
  with $\widehat{\mathcal{F}}_{\overline{\gamma}c}
  =(E_{\overline{\gamma}c}-\mathcal{F}_{\overline{\gamma}c})\oslash\mathcal{F}_{\overline{\gamma}c}$.
  Thus, for any $(\tau,Z)\in\mathbb{X}$, it always holds that
  \[
    \langle (\Delta t,\Delta X),(\tau,Z)\rangle
    -\langle(0,\overline{U}\mathfrak{X}(\Delta\widetilde{\Gamma})\overline{V}^{\mathbb{T}}),(\tau,Z)\rangle
    =\tau \Delta t+{\textstyle\sum_{l=1}^{r_1}}\langle \mathcal{G}(\Delta\widetilde{X}_{a_la_l}),\mathcal{G}(\widetilde{Z}_{a_la_l})\rangle.
  \]
  In view of this, the desired conclusion of this case then follows.

  \medskip
  \noindent
  {\bf(ii)} Let $r_0\in\{1,\ldots,r\}$ be such that $\alpha=\bigcup_{l=1}^{r_0}a_l$
  and $\beta=\bigcup_{l=r_0+1}^{r+1}a_l$, where $\alpha$ and $\beta$ are defined by
  \eqref{abg-case2} with $k_0\in[0,k\!-\!1]$ such that \eqref{sigmabar-case2} holds.
  For any $(\tau,Z)\in\mathbb{X}$, by \eqref{equa3-pdird-case2}
  \begin{align*}
    &\tau\Delta t+\langle \Delta X,Z\rangle
    =\tau\Delta t+\langle \Delta\widetilde{X},\widetilde{Z}\rangle
    =\tau\Delta t +\langle \Delta\widetilde{X}_1,\widetilde{Z}_1\rangle
      +\langle \Delta\widetilde{X}_2,\widetilde{Z}_2\rangle\nonumber\\
    &=\tau \Delta t +\!\langle \mathcal{G}(\Delta\widetilde{X}_1),\mathcal{G}(\widetilde{Z}_1)\rangle
      +\!\langle \mathcal{H}(\Delta\widetilde{X}_1),\mathcal{H}(\widetilde{Z}_1)\rangle
      +\langle \Delta\widetilde{X}_{\alpha c},\widetilde{Z}_{\alpha c}\rangle
      +\langle \Delta\widetilde{X}_{\beta c},\widetilde{Z}_{\beta c}\rangle \nonumber\\
   &=\tau \Delta t+\sum_{l=1}^{r_0}\langle \mathcal{G}(\Delta\widetilde{X}_{a_la_l}),\mathcal{G}(\widetilde{Z}_{a_la_l})\rangle
       +\langle \Delta\widetilde{X}_{\beta\beta},\widetilde{Z}_{\beta\beta}\rangle
       +\langle \Delta\widetilde{X}_{\alpha c},\widetilde{Z}_{\alpha c}\rangle
       +\langle \Delta\widetilde{X}_{\beta c},\widetilde{Z}_{\beta c}\rangle\nonumber\\
    &\quad +2\langle [\mathcal{G}(\Delta\widetilde{X}_1)]_{\alpha\beta},[\mathcal{G}(\widetilde{Z}_1)]_{\alpha\beta}\rangle
           +\langle [\mathcal{H}(\Delta\widetilde{X}_1)]_{\alpha\alpha},[\mathcal{H}(\widetilde{Z}_1)]_{\alpha\alpha}\rangle
          \!+\!2\langle[\mathcal{H}(\Delta\widetilde{X}_1)]_{\alpha\beta},[\mathcal{H}(\widetilde{Z}_1)]_{\alpha\beta}\rangle.
  \end{align*}
  From equation \eqref{equa5-pdird-case2} and equations \eqref{equa7-pdird-case2}-\eqref{equa9-pdird-case2},
  it is immediate to obtain that
  \begin{subnumcases}{}
   \label{temp-equa1-case32}
   \langle [\mathcal{G}(\Delta\widetilde{X}_1)]_{a_la_{l'}},[\mathcal{G}(\widetilde{Z}_1)]_{a_la_{l'}}\rangle
   =\langle (\Theta_1)_{a_la_{l'}}\circ[\mathcal{G}(\Delta\widetilde{\Gamma}_1)]_{a_la_{l'}},
    [\mathcal{G}(\widetilde{Z}_1)]_{a_la_{l'}}\rangle,\\
  \label{temp-equa2-case32}
  \langle [\mathcal{H}(\Delta\widetilde{X}_1)]_{a_la_{l'}},[\mathcal{H}(\widetilde{Z}_1)]_{a_la_{l'}}\rangle
  =\langle(\Theta_2)_{a_la_{l'}}\circ \mathcal{H}(\Delta\widetilde{\Gamma}_{a_la_{l'}}),
  [\mathcal{H}(\widetilde{Z}_1)]_{a_la_{l'}}\rangle,\\
  \langle \Delta\widetilde{X}_{\alpha c},\widetilde{Z}_{\alpha c}\rangle
  =\langle \widehat{\mathcal{F}}_{\alpha c}\circ \Delta\widetilde{\Gamma}_{\alpha c},\widetilde{Z}_{\alpha c}\rangle,\nonumber
  \end{subnumcases}
  where $(\Theta_1)_{a_la_{l'}}=(E_{a_la_{l'}}-(\mathcal{E}_1)_{a_la_{l'}})\oslash(\mathcal{E}_1)_{a_la_{l'}}$
  for $l=1,2,\ldots,r_0$ and $l'=r_0\!+\!1,\ldots,r\!+\!1$,
  $(\Theta_2)_{a_la_{l'}}=(E_{a_la_{l'}}-(\mathcal{E}_2)_{a_la_{l'}})\oslash(\mathcal{E}_2)_{a_la_{l'}}$
  for $l,l'=1,\ldots,r_0$, or $l=1,\ldots,r_0$ and $l'=r_0\!+\!1,\ldots,r\!+\!1$,
  and $\widehat{\mathcal{F}}_{\alpha c}=(E_{\alpha c}-\mathcal{F}_{\alpha c})\oslash \mathcal{F}_{\alpha c}$.
  Let $\mathfrak{X}\!: \mathbb{R}^{m\times n}\to\mathbb{R}^{m\times n}$ be defined by
 \begin{align}\label{MFX-Case2}
   \mathfrak{X}(A)
   &:=\left[\begin{matrix}
     0_{\alpha\alpha} & (\Theta_1)_{\alpha\beta}\circ [\mathcal{G}(A_1)]_{\alpha\beta}
     & \widehat{\mathcal{F}}_{\alpha c}\circ A_{\alpha c}\\
     (\Theta_1)_{\beta\alpha}\circ [\mathcal{G}(A_1)]_{\beta\alpha}
     &0_{\beta\beta}& 0_{\beta c}
    \end{matrix}\right]\nonumber\\
   &\quad+\left[\begin{matrix}
     (\Theta_2)_{\alpha\alpha}\circ \mathcal{H}(A_{\alpha\alpha})
     & (\Theta_2)_{\alpha\beta}\circ [\mathcal{H}(A_1)]_{\alpha\beta}
     & 0_{\alpha c}\\
     (\Theta_2)_{\beta\alpha}\circ [\mathcal{H}(A_1)]_{\beta\alpha}
     &0_{\beta\beta}& 0_{\beta c}
    \end{matrix}\right]\quad \forall A\in\mathbb{R}^{m\times n}.
 \end{align}
 Then, after an elementary calculation, it is easy to calculate that
 \begin{align*}
   &2\langle [\mathcal{G}(\Delta\widetilde{X}_1)]_{\alpha\beta},[\mathcal{G}(\widetilde{Z}_1)]_{\alpha\beta}\rangle
   +\langle [\mathcal{H}(\Delta\widetilde{X}_1)]_{\alpha\alpha},[\mathcal{H}(\widetilde{Z}_1)]_{\alpha\alpha}\rangle \\
   &\!+\!2\langle[\mathcal{H}(\Delta\widetilde{X}_1)]_{\alpha\beta},[\mathcal{H}(\widetilde{Z}_1)]_{\alpha\beta}\rangle
  +\langle \Delta\widetilde{X}_{\alpha c},\widetilde{Z}_{\alpha c}\rangle
  =\langle \mathfrak{X}(\Delta\widetilde{\Gamma}),\widetilde{Z}\rangle
  =\langle\overline{U}\mathfrak{X}(\Delta\widetilde{\Gamma})\overline{V}^{\mathbb{T}},Z\rangle.
 \end{align*}
 Thus, for any $(\tau,Z)\in\mathbb{X}$, it always holds that
 \begin{align*}
    &\langle (\Delta t,\Delta X),(\tau,Z)\rangle
    -\langle(0,\overline{U}\mathfrak{X}(\Delta\widetilde{\Gamma})\overline{V}^{\mathbb{T}}),(\tau,Z)\rangle \\
    &=\tau \Delta t+{\textstyle\sum_{l=1}^{r_0}}\langle \mathcal{G}(\Delta\widetilde{X}_{a_la_l}),\mathcal{G}(\widetilde{Z}_{a_la_l})\rangle
     +\langle\Delta\widetilde{X}_{\beta\beta},\widetilde{Z}_{\beta\beta}\rangle
     +\langle \Delta\widetilde{X}_{\beta c},\widetilde{Z}_{\beta c}\rangle.
 \end{align*}
  By the proof of Case (ii) of Lemma \ref{Lemma1-critical-cone},
  $(\tau,Z)\in\mathcal{T}_K(\overline{t},\overline{X})\cap (\overline{\zeta},\overline{\Gamma})^{\perp}$
  if and only if $[\widetilde{Z}_{\beta\beta}\ \ \widetilde{Z}_{\beta c}]$ has the structure in \eqref{Zbb-case2}.
  For such $(\tau,Z)$, we have that
  $\langle \Delta\widetilde{X}_{\beta c},\widetilde{Z}_{\beta c}\rangle=\langle \Delta\widetilde{X}_{b c},\widetilde{Z}_{bc}\rangle$ and
  \[
    \langle\Delta\widetilde{X}_{\beta\beta},\widetilde{Z}_{\beta\beta}\rangle
    ={\textstyle\sum_{l=r_0+1}^{r}}\langle \mathcal{G}(\Delta\widetilde{X}_{a_la_l}),\mathcal{G}(\widetilde{Z}_{a_la_l})\rangle
     +\langle \widetilde{X}_{bb},\widetilde{Z}_{bb}\rangle
  \]
  The desired result then follows from the last two equalities.
 \end{proof}

%----------------------------------------------------------------------Proof of Proposition 3.1
  \medskip
  \noindent
  {\bf\large The proof of Proposition \ref{main-prop}.} \begin{aproof}
  Notice that $(\overline{t},\overline{X})=\Pi_{K}(t,X),\,
  (\overline{\zeta},\overline{\Gamma})=\Pi_{K^{\circ}}(t,X)$ and
  $\langle (\overline{t},\overline{X}),(\overline{\zeta},\overline{\Gamma})\rangle=0$.
  Let $((\Delta t,\Delta X),(\Delta\zeta,\Delta\Gamma))\in{\rm gph}D\mathcal{N}_{K^\circ} \big((\overline{\zeta},\overline{\Gamma})|(\overline{t},\overline{X})\big)$.
  By the third equivalence of Lemma \ref{tcone-lemma},
  \(
  ((\Delta t,\Delta X),(\Delta\zeta,\Delta\Gamma))\in \mathcal{T}_{K^\circ\times K}\big((\overline{\zeta},\overline{\Gamma}),(\overline{t},\overline{X}))
  \)
  and
    \begin{equation}\label{proj-equa}
    \Pi_{K}'\big((t,X);(\Delta t,\Delta X)+(\Delta\zeta,\Delta\Gamma)\big)
    =(\Delta\zeta,\Delta\Gamma).
  \end{equation}
  Then $(\Delta\zeta,\Delta\Gamma) \in \mathcal{T}_K(\overline{t},\overline{X})$.
  Thus, to prove that $(\Delta\zeta,\Delta\Gamma)\in\mathcal{C}_K(X)$,
  we only need to show that $(\Delta\zeta,\Delta\Gamma) \in (\overline{\zeta},\overline{\Gamma})^{\perp}$.
  We proceed the arguments by three cases as shown below.

  \medskip
  \noindent
  {\bf Case 1:} $(t,X)\in{\rm int}\,K$. Now $(\overline{\zeta},\overline{\Gamma})=(0,0)$
  which implies equation \eqref{equa1-main-prop}, and $\Pi_{K}$ is directionally
  differentiable at $(t,X)$ with $\Pi_{K}'((t,X);(\tau,H))=(\tau,H)$
  for $(\tau,H)\in\mathbb{R}\times\mathbb{R}^{m\times n}$. The latter means that
  $\Pi_{K}'((t,X);(\Delta t,\Delta X)+(\Delta\zeta,\Delta\Gamma))
  =(\Delta t+\Delta\zeta,\Delta X+\Delta\Gamma)$. Together with \eqref{proj-equa},
  we have $(\Delta t,\Delta X)=(0,0)\in[\mathcal{C}_K(\mathcal{X})]^{\circ}$,
  and then \eqref{equa2-main-prop} follows. Equation \eqref{equa3-main-prop} directly
  follows from $(\overline{\zeta},\overline{\Gamma})=(0,0)$ and $(\Delta t,\Delta X)=(0,0)$.

  \medskip
  \noindent
  {\bf Case 2:} $(t,X)\in{\rm int}\,K^{\circ}$. Now $(\overline{t},\overline{X})=(0,0)$
  and $(\overline{\zeta},\overline{\Gamma})=(t,X)\in{\rm int}\,K^{\circ}$.
  Then
  \(
  \mathcal{T}_K(\overline{t},\overline{X})\cap(\overline{\zeta},\overline{\Gamma})^{\perp}
  =K\cap (\overline{\zeta},\overline{\Gamma})^{\perp}=\{(0,0)\}.
  \)
  So, $(\Delta t,\Delta X)\in\mathbb{X}=[\mathcal{C}_K(\mathcal{X})]^{\circ}$,
  and \eqref{equa2-main-prop} follows.
  Notice that $\Pi_{K}$ is directionally differentiable at $(t,X)$
  with $\Pi_{K}'((t,X);(\tau,H))=(0,0)$ for any $(\tau,H)\in\mathbb{R}\times\mathbb{R}^{m\times n}$,
  which means that $\Pi_{K}'((t,X);(\Delta t,\Delta X)+(\Delta\zeta,\Delta\Gamma))=(0,0)$.
  Together with \eqref{proj-equa}, we obtain $(\Delta \zeta,\Delta\Gamma)=(0,0)\in\mathcal{C}_K(\mathcal{X})$,
  and \eqref{equa1-main-prop} follows. Since $(\overline{t},\overline{X})=(0,0)$,
  \eqref{equa3-main-prop} directly follows by $(\Delta \zeta,\Delta\Gamma)=(0,0)$
  and Lemma \ref{Lemma-Upsilon}.

  \medskip
  \noindent
  {\bf Case 3:} $(t,X)\notin{\rm int}\,K\cup{\rm int}\,K^{\circ}$.
  Let $X$ have the SVD as in \eqref{X-SVD} with the index sets
  $a,b,c$ and $a_l\ (l=1,2,\ldots,r)$ defined by \eqref{abc} and \eqref{al}.
  Write $\sigma=\sigma(X)$ and $\overline{\sigma}=\sigma(\overline{X})$.
  Let $\tau=\Delta t+\Delta\zeta$ and $H=\Delta X+\!\Delta\Gamma$.
  We proceed the arguments by two subcases.

  \medskip
  \noindent
  {\bf Subcase 3.1:} $\overline{\sigma}_{k}>0$. Let $r_0,r_1\!\in\{1,2,\ldots,r\}$
  be such that $\alpha=\bigcup_{l=1}^{r_0}a_l$, $\beta=\bigcup_{l=r_0+1}^{r_1}a_l$
  and $\gamma=\bigcup_{l=r_1+1}^{r+1}a_l$, where $\alpha,\beta,\gamma$ and
  $\overline{\gamma}$ are defined by \eqref{abg-case1} with integers
  $k_0\in[0,k\!-\!1]$ and $k_1\in[k,m]$ such that \eqref{sigmabar-case1} holds.
  Also, equation \eqref{WGamma-SVD1} still holds. In addition,
  by \eqref{proj-equa} and Case (i) of Lemma \ref{Lemma2-pdird},
  equations \eqref{equa0-pdird-case1}-\eqref{cequa2-pdird} hold for this case.

  \medskip

  We first prove that equation \eqref{equa1-main-prop} holds.
  By equations \eqref{equa0-pdird-case1}-\eqref{equa1-pdird-case1},
  the definition of the operator $\bm D$ in \eqref{MD-Case1},
  and equation \eqref{Proj-C1set} with $\mathcal{C}_1$ given by \eqref{C1-cone},
  we have that
  \begin{equation}\label{Case31-equa}
    \big(\Delta\zeta,\bm D(\Delta\widetilde{\Gamma})\big)
    =\big(\Phi_0(\tau,\bm D(\widetilde{H})),\Phi_1(\tau,\bm D(\widetilde{H})),\ldots,
    \Phi_{r_1}(\tau,\bm D(\widetilde{H}))\big)
    =\Pi_{\mathcal{C}_1}(\tau,\bm D(\widetilde{H})).
  \end{equation}
  By Lemma \ref{Lemma1-critical-cone}, $(\Delta\zeta,\Delta\Gamma)\in\mathcal{C}_K(\mathcal{X})$,
  i.e., \eqref{equa1-main-prop} follows. By the Moreau decomposition \cite{Moreau65},
  \begin{align*}
    \big(\Delta t+\!\Delta\zeta,\bm D(\Delta\widetilde{X}\!+\!\Delta\widetilde{\Gamma})\big)
    &=\Pi_{\mathcal{C}_1}\big(\Delta t+\!\Delta\zeta,\bm D(\Delta\widetilde{X}\!+\!\Delta\widetilde{\Gamma})\big)
     +\Pi_{\mathcal{C}_1^{\circ}}\big(\Delta t+\Delta\zeta,\bm D(\Delta\widetilde{X}\!+\!\Delta\widetilde{\Gamma})\big)\\
    &=\big(\Delta\zeta,\bm D(\Delta\widetilde{\Gamma})\big)+\Pi_{\mathcal{C}_1^{\circ}}\big(\Delta t+\Delta\zeta,\bm D(\Delta\widetilde{X}+\Delta\widetilde{\Gamma})\big),
  \end{align*}
  which implies that $(\Delta t,\bm D(\Delta\widetilde{X}))
  =\Pi_{\mathcal{C}_1^{\circ}}(\Delta t+\Delta\zeta,\bm D(\Delta\widetilde{X}+\Delta\widetilde{\Gamma}))$.
  Notice that
  \[
    \tau\Delta t +{\textstyle\sum_{l=1}^{r_1}}\big\langle \mathcal{G}(\Delta\widetilde{X}_{a_la_l}),\mathcal{G}(\widetilde{Z}_{a_la_l})\big\rangle
    =\big\langle(\Delta t,\bm D(\Delta\widetilde{X})),(\tau,\bm D(\widetilde{Z}))\big\rangle\le 0
  \]
  for any $(\tau,Z)\in\mathcal{T}_K(\overline{t},\overline{X})\cap(\overline{\zeta},\overline{\Gamma})^{\perp}$,
  where the inequality is due to Lemma \ref{Lemma1-critical-cone}. By Case (i) of Lemma \ref{Lemma2-critical-cone},
  $(\Delta t,\Delta X)-(0,\overline{U}\mathfrak{X}(\Delta\widetilde{\Gamma})\overline{V}^{\mathbb{T}})\in[\mathcal{C}_K(\mathcal{X})]^{\circ}$.
  Thus, equation \eqref{equa2-main-prop} holds.

  \medskip

  We next prove that \eqref{equa3-main-prop} holds.
  By $(\Delta t,\bm D(\Delta\widetilde{X}))
  =\Pi_{\mathcal{C}_1^{\circ}}(\Delta t+\Delta\zeta,\bm D(\Delta\widetilde{X}+\Delta\widetilde{\Gamma}))$,
  \begin{align*}
  &\Delta t\Delta\zeta+{\textstyle\sum_{l=1}^{r_1}}\big\langle \mathcal{G}(\Delta\widetilde{X}_{a_la_l}),\mathcal{G}(\Delta\widetilde{\Gamma}_{a_la_l})\big\rangle
  = \big\langle (\Delta t,\bm D(\Delta\widetilde{X})),(\Delta\zeta,\bm D(\Delta\widetilde{\Gamma}))\big\rangle\nonumber\\
  =&\ \langle \Pi_{\mathcal{C}_1^{\circ}}(\Delta t+\Delta\zeta,\bm D\big(\Delta\widetilde{X}+\Delta\widetilde{\Gamma})\big),
  \Pi_{\mathcal{C}_1}\big(\Delta t+\Delta\zeta,\bm D(\Delta\widetilde{X}+\Delta\widetilde{\Gamma})\big)\rangle= 0.
 \end{align*}
  Together with the arguments for Case (i) of Lemma \ref{Lemma2-critical-cone},
  we have that
  \begin{align}\label{main-equa1}
   &\langle(\Delta t,\Delta X),(\Delta\zeta,\Delta\Gamma)\rangle
    =\langle \mathfrak{X}(\Delta\widetilde{\Gamma}),\Delta\widetilde{\Gamma}\rangle\nonumber\\
   &= \langle \mathcal{H}(\Delta\widetilde{X}_1),\mathcal{H}(\Delta\widetilde{\Gamma}_1)\rangle
      \!+\!\langle\Delta\widetilde{X}_{\overline{\gamma}c},\Delta\widetilde{\Gamma}_{\overline{\gamma}c}\rangle
      +2\langle [\mathcal{G}(\Delta\widetilde{X}_1)]_{\alpha\beta},[\mathcal{G}(\Delta\widetilde{\Gamma}_1)]_{\alpha\beta}\rangle\nonumber\\
   &\quad +2\big[\langle[\mathcal{G}(\Delta\widetilde{X}_1)]_{\alpha\gamma},[\mathcal{G}(\Delta\widetilde{\Gamma}_1)]_{\alpha\gamma}\rangle
          \!+\!\langle [\mathcal{G}(\Delta\widetilde{X}_1)]_{\beta\gamma},
          [\mathcal{G}(\Delta\widetilde{\Gamma}_1)]_{\beta\gamma}\rangle\big].
  \end{align}
  Next we consider each term on the right hand side of \eqref{main-equa1} separately.
  By \eqref{cequa1-pdird}, we have
  \begin{equation}\label{temp-equa1}
  \langle\Delta\widetilde{X}_{\overline{\gamma}c},\Delta\widetilde{\Gamma}_{\overline{\gamma}c}\rangle
   =\sum_{l=1}^{r_0}\frac{\theta}{\overline{\nu}_l}\|\Delta\widetilde{\Gamma}_{a_lc}\|^2
   +\sum_{l=r_0+1}^{r_1}\frac{\theta\overline{u}_{l'}}{\overline{\nu}}\|\Delta\widetilde{\Gamma}_{a_{l'}c}\|^2.
  \end{equation}
  While from equations \eqref{equa4-pdird-case1}-\eqref{equa6-pdird-case1} it is not difficult to obtain that
  \begin{numcases}{}
  \langle [\mathcal{G}(\Delta\widetilde{X}_1)]_{\alpha\beta},[\mathcal{G}(\Delta\widetilde{\Gamma}_1)]_{\alpha\beta}\rangle
  =\sum_{l=1}^{r_0}\sum_{l'=r_0+1}^{r_1}\frac{\theta-\theta\overline{u}_{l'}}{\overline{\nu}_l-\overline{\nu}_{l'}}
   \|[\mathcal{G}(\Delta\widetilde{\Gamma}_1)]_{a_la_{l'}}\|^2,\nonumber\\
   \langle[\mathcal{G}(\Delta\widetilde{X}_1)]_{\alpha\gamma},[\mathcal{G}(\Delta\widetilde{\Gamma}_1)]_{\alpha\gamma}\rangle
 =\sum_{l=1}^{r_0}\sum_{l'=r_1+1}^{r+1}\frac{\theta}{\overline{\nu}_l-\overline{\nu}_{l'}}
   \|[\mathcal{G}(\Delta\widetilde{\Gamma}_1)]_{a_la_{l'}}\|^2,\nonumber\\
  \langle [\mathcal{G}(\Delta\widetilde{X}_1)]_{\beta\gamma},[\mathcal{G}(\Delta\widetilde{\Gamma}_1)]_{\beta\gamma}\rangle
  =\sum_{l=r_0+1}^{r_1}\sum_{l'=r_1+1}^{r+1}\frac{\theta\overline{u}_{l}}{\overline{\nu}_l-\overline{\nu}_{l'}}\nonumber
   \|[\mathcal{G}(\Delta\widetilde{\Gamma}_1)]_{a_la_{l'}}\|^2.\nonumber
  \end{numcases}{}
  Adding the last three equalities together and making suitable rearrangement yields that
  \begin{align}\label{temp-equa2}
  &2\big[\langle [\mathcal{G}(\Delta\widetilde{X}_1)]_{\alpha\beta},[\mathcal{G}(\Delta\widetilde{\Gamma}_1)]_{\alpha\beta}\rangle
   +\langle[\mathcal{G}(\Delta\widetilde{X}_1)]_{\alpha\gamma},[\mathcal{G}(\Delta\widetilde{\Gamma}_1)]_{\alpha\gamma}\rangle
   +\langle [\mathcal{G}(\Delta\widetilde{X}_1)]_{\beta\gamma},[\mathcal{G}(\Delta\widetilde{\Gamma}_1)]_{\beta\gamma}\rangle\big]\nonumber\\
  =&\ -2\sum_{l=1}^{r_0}\sum_{l'=r_0+1}^{r+1}\frac{\theta}{\overline{\nu}_{l'}-\overline{\nu}_l}
       \|[\mathcal{G}(\Delta\widetilde{\Gamma}_1)]_{a_la_{l'}}\|^2\nonumber\\
  &\ -\!2\!\sum_{l'=r_0+1}^{r_1}\Big(\sum_{l=1}^{r_0}\frac{\theta\overline{u}_{l'}}{\overline{\nu}_l-\overline{\nu}_{l'}}
   \|[\mathcal{G}(\Delta\widetilde{\Gamma}_1)]_{a_la_{l'}}\|^2
        +\sum_{l'=r_1+1}^{r+1}\frac{\theta\overline{u}_{l}}{\overline{\nu}_{l'}-\overline{\nu}_l}
   \|[\mathcal{G}(\Delta\widetilde{\Gamma}_1)]_{a_la_{l'}}\|^2\Big).
  \end{align}
  In addition, by using equation \eqref{equa7-pdird-case1}, we can calculate that
  \begin{align}\label{temp-equa3}
   &\quad \langle\mathcal{H}(\Delta\widetilde{X}_1),\mathcal{H}(\Delta\widetilde{\Gamma}_1)\rangle
   =\langle \mathcal{H}(\Delta\widetilde{X}_{\alpha\alpha}),\mathcal{H}(\Delta\widetilde{\Gamma}_{\alpha\alpha})\rangle
    +2\langle [\mathcal{H}(\Delta\widetilde{X}_1)]_{\alpha\beta},[\mathcal{H}(\Delta\widetilde{\Gamma}_1)]_{\alpha\beta}\rangle\nonumber\\
   &\quad +\langle \mathcal{H}(\Delta\widetilde{X}_{\beta\beta}),\mathcal{H}(\Delta\widetilde{\Gamma}_{\beta\beta})\rangle
       +2\langle [\mathcal{H}(\Delta\widetilde{X}_1)]_{\alpha\gamma},[\mathcal{H}(\Delta\widetilde{\Gamma}_1)]_{\alpha\gamma}\rangle
        +2\langle[\mathcal{H}(\Delta\widetilde{X}_1)]_{\beta\gamma},[\mathcal{H}(\Delta\widetilde{\Gamma}_1)]_{\beta\gamma}\rangle\nonumber\\
   &=2\sum_{l=1}^{r_0}\sum_{l'=1}^{r_0}\frac{\theta}{\overline{\nu}_l+\overline{\nu}_{l'}}\|[\mathcal{H}(\Delta\widetilde{\Gamma}_1)]_{a_la_{l'}}\|^2
     +2\sum_{l=r_0+1}^{r_1}\sum_{l'=r_0+1}^{r_1}\frac{\theta\overline{u}_l}{2\overline{\nu}}\|[\mathcal{H}(\Delta\widetilde{\Gamma}_1)]_{a_la_{l'}}\|^2\nonumber\\
   &\quad +2\!\sum_{l=1}^{r_0}\sum_{l'=r_0+1}^{r_1}\frac{\theta+\theta\overline{u}_{l'}}{\overline{\nu}_l+\overline{\nu}}
          \|[\mathcal{H}(\Delta\widetilde{\Gamma}_1)]_{a_la_{l'}}\|^2
   +2\!\sum_{l=1}^{r_0}\sum_{l'=r_1+1}^{r+1}\frac{\theta}{\overline{\nu}_l+\overline{\nu}_{l'}}
       \big\|[\mathcal{H}(\Delta\widetilde{\Gamma}_1)]_{a_la_{l'}}\big\|^2\nonumber\\
   &\quad +2\sum_{l=r_0+1}^{r_1}\sum_{l'=r_1+1}^{r+1}\frac{\theta\overline{u}_l}{\overline{\nu}+\overline{\nu}_{l'}}
          \big\|[\mathcal{H}(\Delta\widetilde{\Gamma}_1)]_{a_la_{l'}}\big\|^2\nonumber\\
   &= 2\sum_{l=1}^{r_0}\sum_{l'=1}^{r+1}\frac{\theta}{\overline{\nu}_l+\overline{\nu}_{l'}}
      \|[\mathcal{H}(\Delta\widetilde{\Gamma}_1)]_{a_la_{l'}}\|^2
     +2\sum_{l=r_0+1}^{r_1}\sum_{l'=1}^{r+1}\frac{\theta\overline{u}_l}{\overline{\nu}_{l'}+\overline{\nu}}
      \|[\mathcal{H}(\Delta\widetilde{\Gamma}_1)]_{a_la_{l'}}\|^2.
  \end{align}
  Substituting equations \eqref{temp-equa1}-\eqref{temp-equa3} into equation \eqref{main-equa1} immediately yields that
  \begin{align}\label{one-hand}
  &\langle(\Delta t,\Delta X),(\Delta\zeta,\Delta\Gamma)\rangle
  =\langle \mathfrak{X}(\Delta\widetilde{\Gamma}),\Delta\widetilde{\Gamma}\rangle\nonumber\\
  &\!=2\sum_{l=1}^{r_0}\sum_{l'=1}^{r+1}\frac{\theta}{\overline{\nu}_l+\overline{\nu}_{l'}}
       \|[\mathcal{H}(\Delta\widetilde{\Gamma}_1)]_{a_la_{l'}}\|^2
      -2\sum_{l=1}^{r_0}\sum_{l'=r_0+1}^{r+1}\frac{\theta}{\overline{\nu}_{l'}-\overline{\nu}_l}
       \|[\mathcal{G}(\Delta\widetilde{\Gamma}_1)]_{a_la_{l'}}\|^2 \nonumber\\
  &\quad +\!\sum_{l=1}^{r_0}\frac{\theta}{\overline{\nu}_l}\|\Delta\widetilde{\Gamma}_{a_lc}\|^2
        +2\!\sum_{l=r_0+1}^{r_1}\sum_{l'=1}^{r+1}\frac{\theta\overline{u}_l}{\overline{\nu}_{l'}+\overline{\nu}}
        \|[\mathcal{H}(\Delta\widetilde{\Gamma}_1)]_{a_la_{l'}}\|^2
       +\sum_{l=r_0+1}^{r_1}\frac{\theta\overline{u}_{l'}}{\overline{\nu}}\|\Delta\widetilde{\Gamma}_{a_{l'}c}\|^2\nonumber\\
  &\quad -\!2\!\sum_{l'=r_0+1}^{r_1}\!\Big[\sum_{l=1}^{r_0}\frac{\theta\overline{u}_{l'}}{\overline{\nu}_l-\!\overline{\nu}}
       \|[\mathcal{G}(\Delta\widetilde{\Gamma}_1)]_{a_la_{l'}}\|^2
        \!+\!\sum_{l=r_1+1}^{r+1}\!\frac{\theta\overline{u}_{l}}{\overline{\nu}-\!\overline{\nu}_l}
        \|[\mathcal{G}(\Delta\widetilde{\Gamma}_1)]_{a_la_{l'}}\|^2\Big].
  \end{align}
  On the other hand, by \eqref{WGamma-SVD1} and the definitions of $\mathcal{B}$ and $\overline{P}$
  in Lemma \ref{Lemma-Upsilon}, we have that
  \begin{align*}
   &\overline{\zeta}\sum_{j=1}^{r_0}{\rm tr}\Big(\overline{P}_{\!a_j}^{\mathbb{T}}
     \big[\mathcal{B}(\Delta\Gamma)(\mathcal{B}(\overline{X})\!-\overline{\nu}_jI)^\dag\mathcal{B}(\Delta\Gamma)\big]\overline{P}_{\!a_j}\Big)\\
   &=\sum_{l=1}^{r_0}\sum_{l'=r_0+1}^{r+1}\!\frac{\theta}{\overline{\nu}_{l'}\!-\!\overline{\nu}_l} \|[\mathcal{G}(\Delta\widetilde{\Gamma}_1)]_{a_la_{l'}}\|^2
    -\!\frac{1}{2}\sum_{l=1}^{r_0}\frac{\theta}{\overline{\nu}_l}\|\Delta\widetilde{\Gamma}_{a_lc}\|^2
     -\!\sum_{l=1}^{r_0}\sum_{l'=1}^{r+1}\frac{\theta}{\overline{\nu}_l\!+\!\overline{\nu}_{l'}}
      \|[\mathcal{H}(\Delta\widetilde{\Gamma}_1)]_{a_la_{l'}}\|^2\nonumber
  \end{align*}
  and
  \begin{align*}
   & \Big\langle\Sigma_{\beta\beta}(\overline{\Gamma}),
       \overline{P}_{\!\beta}^{\mathbb{T}}\mathcal{B}(\Delta\Gamma)(\mathcal{B}(\overline{X})-\overline{\nu}I)^\dag\mathcal{B}(\Delta\Gamma)
       \overline{P}_{\!\beta}\Big\rangle\nonumber\\
   &= -\sum_{l=r_0+1}^{r_1}\sum_{l'=1}^{r+1}\frac{\theta\overline{u}_l}{\overline{\nu}_{l'}+\overline{\nu}}
    \|[\mathcal{H}(\Delta\widetilde{\Gamma}_1)]_{a_la_{l'}}\|^2
   -\frac{1}{2}\sum_{l=r_0+1}^{r_1}\frac{\theta\overline{u}_{l'}}{\overline{\nu}}\|\Delta\widetilde{\Gamma}_{a_{l'}c}\|^2,\\
   &\quad+\!\sum_{l'=r_0+1}^{r_1}\Big(\sum_{l=1}^{r_0}\frac{\theta\overline{u}_{l'}}{\overline{\nu}_l-\overline{\nu}}
       \|[\mathcal{G}(\Delta\widetilde{\Gamma}_1)]_{a_la_{l'}}\|^2
        +\sum_{l=r_1+1}^{r+1}\frac{\theta\overline{u}_{l}}{\overline{\nu}-\overline{\nu}_l}
        \|[\mathcal{G}(\Delta\widetilde{\Gamma}_1)]_{a_la_{l'}}\|^2\Big).
  \end{align*}
  By substituting the last two equalities into the expression of
  $\Upsilon_{(\overline{t},\overline{X})}\big((\overline{\zeta},\overline{\Gamma}),(\Delta\zeta,\Delta\Gamma)\big)$
  and then comparing with equation \eqref{one-hand}, we obtain equation \eqref{equa3-main-prop}.

  \medskip
  \noindent
  {\bf Subcase 3.2:} $\overline{\sigma}_k=0$. Let $r_0\in\{1,\ldots,r\}$ be such that
  $\alpha=\bigcup_{l=1}^{r_0}a_l$ and $\beta=\bigcup_{l=r_0+1}^{r+1}a_l$,
  where $\alpha$ and $\beta$ are defined by \eqref{abg-case2} with $k_0\in[0,k\!-\!1]$
  such that equation \eqref{sigmabar-case2} holds. Also, equation \eqref{WGamma-SVD2} holds.
  By \eqref{proj-equa} and Lemma \ref{Lemma2-pdird},
  equations \eqref{equa0-pdird-case2}-\eqref{equa9-pdird-case2} hold.

  \medskip

  We first prove that equation \eqref{equa1-main-prop} holds.
  By equations \eqref{equa0-pdird-case2}-\eqref{equa2-pdird-case2},
  the definition of the operator $\bm D$ in \eqref{MD-Case2},
  and equation \eqref{Proj-C2set} with $\mathcal{C}_2$ given by \eqref{C2-cone},
  we have that
  \begin{equation}\label{Case32-equa}
   (\Delta\zeta,\bm D(\Delta\widetilde{\Gamma}))\!=\!
   \big(\Phi_{0}(\tau,\bm D(\widetilde{H})),\Phi_{1}(\tau,\bm D(\widetilde{H})),\ldots,
   \Phi_{r+1}(\tau,\bm D(\widetilde{H}))\big)
   \!=\Pi_{\mathcal{C}_2}(\tau,\bm D(\widetilde{H})).
  \end{equation}
  By Lemma \ref{Lemma1-critical-cone}, $(\Delta\zeta,\Delta\Gamma)\in\mathcal{C}_K(\mathcal{X})$.
  i.e., \eqref{equa1-main-prop} follows. By the Moreau decomposition \cite{Moreau65},
  \[
   \big(\Delta t+\!\Delta\zeta,\bm D(\Delta\widetilde{X}\!+\!\Delta\widetilde{\Gamma})\big)
    =\big(\Delta\zeta,\bm D(\Delta\widetilde{\Gamma})\big)+\Pi_{\mathcal{C}_2^{\circ}}\big(\Delta t+\Delta\zeta,\bm D(\Delta\widetilde{X}+\Delta\widetilde{\Gamma})\big),
  \]
  which implies that
  $(\Delta t,\bm D(\Delta\widetilde{X}))
  =\Pi_{\mathcal{C}_2^{\circ}}(\Delta t+\Delta\zeta,\bm D(\Delta\widetilde{X}+\Delta\widetilde{\Gamma}))$.
  Notice that
  \begin{align*}
    &\tau\Delta t +{\textstyle\sum_{l=1}^r}\big\langle \mathcal{G}(\Delta\widetilde{X}_{a_la_l}),\mathcal{G}(\widetilde{Z}_{a_la_l})\big\rangle
    +\big\langle[\Delta\widetilde{X}_{bb}\ \  \widetilde{Z}_{bc}],[\Delta\widetilde{X}_{bc}\ \ \widetilde{Z}_{bc}]\big\rangle\\
    &=\big\langle(\Delta t,\bm D(\Delta\widetilde{X})),(\tau,\bm D(\widetilde{Z}))\big\rangle\le 0,
  \end{align*}
  for any $(\tau,Z)\in\mathcal{T}_K(\overline{t},\overline{X})\cap(\overline{\zeta},\overline{\Gamma})^{\perp}$,
  where the inequality is using Lemma \ref{Lemma1-critical-cone}. By Case (ii) of Lemma \ref{Lemma2-critical-cone},
  $(\Delta t,\Delta X)-(0,\overline{U}\mathfrak{X}(\Delta\widetilde{\Gamma})\overline{V}^{\mathbb{T}})\in[\mathcal{C}_K(\mathcal{X})]^{\circ}$.
  Thus, equation \eqref{equa2-main-prop} holds.

  \medskip

  We next prove that \eqref{equa3-main-prop} holds.
  By $(\Delta t,\bm D(\Delta\widetilde{X}))
  =\Pi_{\mathcal{C}_2^{\circ}}(\Delta t+\Delta\zeta,\bm D(\Delta\widetilde{X}+\Delta\widetilde{\Gamma}))$,
  \begin{align*}
   &\Delta t\Delta\zeta
   +{\textstyle\sum_{l=1}^r}\big\langle[\mathcal{G}(\Delta\widetilde{X}_1)]_{a_la_l},[\mathcal{G}(\Delta\widetilde{\Gamma}_1)]_{a_la_l}\big\rangle
   +\big\langle[\Delta\widetilde{X}_{bb}\ \ \Delta\widetilde{X}_{bc}],[\Delta\widetilde{X}_{bb}\ \ \Delta\widetilde{\Gamma}_{bc}]\big\rangle\nonumber\\
   =&\ \big\langle (\Delta t,\bm D(\Delta\widetilde{X})),(\Delta\zeta,\bm D(\Delta\widetilde{\Gamma}))\big\rangle
  =\big\langle\Pi_{\mathcal{C}_2^{\circ}}(\tau,\bm D(\widetilde{H})),\Pi_{\mathcal{C}_2}(\tau,\bm D(\widetilde{H}))\big\rangle
  =0.
 \end{align*}
  Together with the arguments for Case (ii) of Lemma \ref{Lemma2-critical-cone} and
  equation \eqref{equa6-pdird-case2},
  we know that
  \begin{align}\label{main-equa1-case2}
    &\langle(\Delta t,\Delta X),(\Delta\zeta,\Delta\Gamma)\rangle
     =\langle \mathfrak{X}(\Delta\widetilde{\Gamma}),\Delta\widetilde{\Gamma}\rangle\nonumber\\
   &= 2\langle[\mathcal{G}(\Delta\widetilde{X}_1)]_{\alpha\beta},[\mathcal{G}(\Delta\widetilde{\Gamma}_1)]_{\alpha\beta}\rangle
       +\langle \mathcal{H}(\Delta\widetilde{X}_{\alpha\alpha}),\mathcal{H}(\Delta\widetilde{\Gamma}_{\alpha\alpha})\rangle\nonumber\\
   &\quad +2\langle [\mathcal{H}(\Delta\widetilde{X}_1)]_{\alpha\beta},[\mathcal{H}(\Delta\widetilde{\Gamma}_1)]_{\alpha\beta}\rangle
       +\langle\Delta\widetilde{X}_{\alpha c},\Delta\widetilde{\Gamma}_{\alpha c}\rangle.
  \end{align}
  We next consider each term on the right hand side of \eqref{main-equa1-case2} separately.
  By \eqref{equa5-pdird-case2}, we have
  \begin{align*}
  &2\langle [\mathcal{G}(\Delta\widetilde{X}_1)]_{\alpha\beta},[\mathcal{G}(\Delta\widetilde{\Gamma}_1)]_{\alpha\beta}\rangle \\
  &=\!2\sum_{l=1}^{r_0}\!\sum_{l'=r_0+1}^{r+1}\!\frac{\theta}{\overline{\nu}_l}\|[\mathcal{G}(\Delta\widetilde{\Gamma}_1)]_{a_la_{l'}}\|^2
   \!-\!2\!\sum_{l'=r_0+1}^{r+1}\!\sum_{l=1}^{r_0}\!\frac{\theta\overline{u}_{l'}}{\overline{\nu}_l}
     \|[\mathcal{G}(\Delta\widetilde{\Gamma}_1)]_{a_la_{l'}}\|^2,
  \end{align*}
  while from equations \eqref{equa7-pdird-case2} and \eqref{equa8-pdird-case2}
  it is not difficult to obtain that
  \begin{align*}
  &\langle \mathcal{H}(\Delta\widetilde{X}_{\alpha\alpha}),\mathcal{H}(\Delta\widetilde{\Gamma}_{\alpha\alpha})\rangle
  +2\langle [\mathcal{H}(\Delta\widetilde{X}_1)]_{\alpha\beta},[\mathcal{H}(\Delta\widetilde{\Gamma}_1)]_{\alpha\beta}\rangle \nonumber\\
  =&\ 2\sum_{l=1}^{r_0}\!\sum_{l'=1}^{r_0}\!\frac{\theta}{\overline{\nu}_{l}+\overline{\nu}_{l'}}
    \|[\mathcal{H}(\Delta\widetilde{\Gamma}_1)]_{a_la_{l'}}\|^2
   \!+\!2\!\sum_{l=1}^{r_0}\sum_{l'=r_0+1}^{r+1}\!\frac{\theta+\theta\overline{u}_{l'}}{\overline{\nu}_l}
      \|[\mathcal{H}(\Delta\widetilde{\Gamma}_1)]_{a_la_{l'}}\|^2.
  \end{align*}
  In addition, by using \eqref{equa9-pdird-case2}, we can easily obtain that
  \(
    \langle\Delta\widetilde{X}_{\alpha c},\Delta\widetilde{\Gamma}_{\alpha c}\rangle
    =\sum_{l=1}^{r_0}\frac{\theta}{\overline{\nu}_l}\|\Delta\widetilde{\Gamma}_{a_lc}\|^2.
  \)
  Substituting the three equalities into equation \eqref{main-equa1-case2} then yields that
  \begin{align}\label{one-hand-case2}
    &\langle(\Delta t,\Delta X),(\Delta\zeta,\Delta\Gamma)\rangle \nonumber\\
    =&\ \!2\sum_{l=1}^{r_0}\!\sum_{l'=r_0+1}^{r+1}\!\frac{\theta}{\overline{\nu}_l}\|[\mathcal{G}(\Delta\widetilde{\Gamma}_1)]_{a_la_{l'}}\|^2
   \!-\!2\!\sum_{l'=r_0+1}^{r+1}\!\sum_{l=1}^{r_0}\!\frac{\theta\overline{u}_{l'}}{\overline{\nu}_l}
     \|[\mathcal{G}(\Delta\widetilde{\Gamma}_1)]_{a_la_{l'}}\|^2
   +\!\sum_{l=1}^{r_0}\frac{\theta}{\overline{\nu}_l}\|\Delta\widetilde{\Gamma}_{a_lc}\|^2\nonumber\\
   &+2\sum_{l=1}^{r_0}\!\sum_{l'=1}^{r_0}\!\frac{\theta}{\overline{\nu}_{l}+\overline{\nu}_{l'}}
     \|[\mathcal{H}(\Delta\widetilde{\Gamma}_1)]_{a_la_{l'}}\|^2
    +2\!\sum_{l=1}^{r_0}\sum_{l'=r_0+1}^{r+1}\!\frac{\theta+\theta\overline{u}_{l'}}{\overline{\nu}_l}
    \|[\mathcal{H}(\Delta\widetilde{\Gamma}_1)]_{a_la_{l'}}\|^2.
  \end{align}
  On the other hand, by using $\zeta=-\theta$, $\overline{X}=\overline{U}[{\rm Diag}(\overline{\sigma})\ \ 0]\overline{V}^{\mathbb{T}}$
  and the definitions of $\mathcal{B}$ and $\overline{P}$ in Lemma \ref{Lemma-Upsilon},
  we calculate that
  \begin{align*}
   &-2\overline{\zeta}\sum_{j=1}^{r_0}{\rm tr}\Big(\overline{P}_{\!a_j}^{\mathbb{T}}
     \big[\mathcal{B}(\Delta\Gamma)(\mathcal{B}(\overline{X})-\overline{\nu}_jI)^\dag\mathcal{B}(\Delta\Gamma)\big]\overline{P}_{\!a_j}\Big)\\
   =&\ 2\sum_{l=1}^{r_0}\sum_{l'=r_0+1}^{r+1}\!\frac{\theta}{-\!\overline{\nu}_l}
      \|[\mathcal{G}(\Delta\widetilde{\Gamma}_1)]_{a_la_{l'}}\|^2
     -\!2\sum_{l=1}^{r_0}\sum_{l'=1}^{r+1}\frac{\theta}{\overline{\nu}_l+\overline{\nu}_{l'}}
     \|[\mathcal{H}(\Delta\widetilde{\Gamma}_1)]_{a_la_{l'}}\|^2
    -\!\sum_{l=1}^{r_0}\frac{\theta}{\overline{\nu}_l}\|\Delta\widetilde{\Gamma}_{a_lc}\|^2.\nonumber
  \end{align*}
  Notice that $\langle[\mathcal{G}(\Delta\widetilde{X}_1)]_{a_la_{l'}},[\mathcal{H}(\Delta\widetilde{X}_1)]_{a_la_{l'}}\rangle=0$
  for $l\in\{1,\ldots,r_0\}$ and $l'\in\{r_0\!+\!1,\ldots,r\!+\!1\}$. By this, we may calculate that
  \begin{align*}
  & 2\left\langle \big[\Sigma_{\beta\beta}(\overline{\Gamma})\ \ 0\big],
    \big[\overline{U}_{\!\beta}^{\mathbb{T}}\Delta X \overline{X}^{\dag}\Delta X\overline{V}_{\beta}\ \ \overline{U}_{\!\beta}^{\mathbb{T}}\Delta X \overline{X}^{\dag}\Delta X\overline{V}_{2}\big]\right\rangle\nonumber\\
  =&\ 2\sum_{l'=r_0+1}^{r+1}\sum_{l=1}^{r_0}\frac{\theta\overline{u}_{l'}}{\overline{\nu}_l}
    \|[\mathcal{G}(\Delta\widetilde{\Gamma}_1)]_{a_la_{l'}}\|^2
   +2\sum_{l'=r_0+1}^{r+1}\sum_{l=1}^{r_0}\frac{\theta\overline{u}_{l'}}{-\overline{\nu}_l}
    \|[\mathcal{H}(\Delta\widetilde{\Gamma}_1)]_{a_la_{l'}}\|^2.
  \end{align*}
  Substituting the last two equalities into the expression of
  $\Upsilon_{(\overline{t},\overline{X})}\big((\overline{\zeta},\overline{\Gamma}),(\Delta\zeta,\Delta\Gamma)\big)$
  and then comparing with equation \eqref{one-hand-case2},
  we obtain equation \eqref{equa3-main-prop}.

  \medskip

  Conversely, suppose that \eqref{equa1-main-prop} and \eqref{equa3-main-prop} hold.
  By Lemma \ref{tcone-lemma} it suffices to show that equation \eqref{proj-equa} holds.
  If $\mathcal{X}\in{\rm int}\,K$, then $(\Delta t,\Delta X)=(0,0)$ since
  $[\mathcal{C}_K(\mathcal{X})]^{\circ}=\{(0,0)\}$.
  If $\mathcal{X}\in{\rm int}\,K^{\circ}$, then $(\Delta \zeta,\Delta \Gamma)=(0,0)$
  since $\mathcal{C}_K(\mathcal{X})=\{(0,0)\}$. Together with the directional
  derivative of $\Pi_K$ in the two cases, we obtain equation \eqref{proj-equa}.
  Next we assume that $\mathcal{X}\notin{\rm int}\,K\cup{\rm int}\,K^{\circ}$.
  Then it must hold that
  \begin{equation}\label{inner-product-equa}
   \big\langle (\Delta t,\Delta X)-\widehat{\mathfrak{X}}((\Delta\zeta,\Delta\Gamma)),(\Delta\zeta,\Delta\Gamma)\big\rangle=0,
  \end{equation}
  where $\widehat{\mathfrak{X}}\!:\mathbb{X}\to\mathbb{X}$ is defined by
  $\widehat{\mathfrak{X}}(\omega,W):=(0,\overline{U}\mathfrak{X}(\overline{U}^{\mathbb{T}}W\overline{V})\overline{V}^{\mathbb{T}})$
  for any $(\omega,W)\in\mathbb{X}$,
  Indeed, from the previous arguments in Subcase 3.1 and Subcase 3.2, we know that
  \begin{align}\label{sigma-term}
   \langle\widehat{\mathfrak{X}}((\Delta\zeta,\Delta\Gamma)),(\Delta\zeta,\Delta\Gamma)\rangle
   &=\langle(0,\overline{U}\mathfrak{X}(\Delta\widetilde{\Gamma})\overline{V}^{\mathbb{T}}),(\Delta\zeta,\Delta\Gamma)\rangle
   =-\Upsilon_{(\overline{t},\overline{X})}((\overline{\zeta},\overline{\Gamma}),(\Delta\zeta,\Delta\Gamma))\nonumber\\
   &=-\sigma\big((\overline{\zeta},\overline{\Gamma}),\mathcal{T}_K^2((\overline{t},\overline{X}),(\Delta\zeta,\Delta\Gamma))\big).
  \end{align}
  Along with \eqref{equa3-main-prop}, we obtain equality \eqref{inner-product-equa}.
  Combining  \eqref{inner-product-equa} with \eqref{equa1-main-prop}-\eqref{equa2-main-prop} yields that
  \[
    \left[\left(\begin{matrix}
     \Delta t \\ \Delta X
     \end{matrix}\right)+\left(\begin{matrix}
     \Delta \zeta\\ \Delta\Gamma
     \end{matrix}\right)\right]
     -\left(\begin{matrix}
     \Delta \zeta\\ \Delta\Gamma
     \end{matrix}\right)-
     \widehat{\mathfrak{X}}\left(\begin{matrix}
    \Delta \zeta\\ \Delta\Gamma
     \end{matrix}\right)\in\mathcal{N}_{\mathcal{C}_K(\mathcal{X})}
     \left(\begin{matrix}
     \Delta \zeta\\ \Delta\Gamma
     \end{matrix}\right).
  \]
  Together with \eqref{sigma-term}, using \cite[Theorem 7.2]{BS98} with
  $d=(\Delta t+\Delta\zeta,\Delta X+\Delta\Gamma)$ and $u^*=(\Delta\zeta,\Delta\Gamma)$
  shows that equation \eqref{proj-equa} holds. The proof is then completed.
 \end{aproof}
 \end{document}